\documentclass[11pt, oneside]{amsart}
\usepackage{geometry} 
\usepackage{graphicx}
\usepackage{amssymb, amsthm}
\usepackage[all]{xy}
\usepackage[dvipsnames]{xcolor}
\usepackage{comment}
\usepackage{pinlabel}
\usepackage[inline]{enumitem}
\usepackage{todonotes}
\usepackage{thm-restate}
\usepackage{verbatim}
\usepackage{tikz}
\usepackage{soul}
\usepackage{hyperref}
\usepackage[nameinlink]{cleveref}
\usepackage{cite}

\hypersetup{colorlinks=true,linkcolor=teal,citecolor=purple}

\newtheorem{theorem}{Theorem}[section]
\newtheorem{lemma}[theorem]{Lemma}
\newtheorem{proposition}[theorem]{Proposition}
\newtheorem{corollary}[theorem]{Corollary}

\theoremstyle{definition}
\newtheorem{definition}[theorem]{Definition}
\newtheorem{remark}[theorem]{Remark}
\newtheorem{example}[theorem]{Example}
\newtheorem{question}[theorem]{Question}

\newcommand{\C}{{\ensuremath{\mathbb{C}}}}
\newcommand{\R}{{\ensuremath{\mathbb{R}}}}

\newcommand{\Z}{{\ensuremath{\mathbb{Z}}}}
\newcommand{\Q}{{\ensuremath{\mathbb{Q}}}}
\newcommand{\HFhat}{\widehat{HF}}
\newcommand{\Isharp}{I^{\#}}
\newcommand{\maxtb}{\overline{tb}}
\newcommand{\maxsl}{\overline{sl}}

\subjclass[2013]{}

\author[Tye Lidman]{Tye Lidman}
\thanks{The first author was partially supported by NSF grant DMS-1709702 and a Sloan Fellowship.}
\address{Department of Mathematics, North Carolina State University, Raleigh, NC 27607}
\email{tlid@math.ncsu.edu}

\author[Juanita Pinz\'on-Caicedo]{Juanita Pinz\'on-Caicedo}
\thanks{The second author was partially supported by NSF grant DMS-1664567.}
\address {Department of Mathematics, University of Notre Dame, Notre Dame, IN 46556}
\email{jpinzonc@nd.edu}

\author[Christopher Scaduto]{Christopher Scaduto}
\thanks{}
\address{Department of Mathematics, University of Miami, Coral Gables, FL 33146}
\email{c.scaduto@math.miami.edu}

\numberwithin{equation}{section}

\title[Framed instanton homology of surgeries on L-space knots]{Framed instanton homology of\\ surgeries on L-space knots}

\begin{document}
\begin{abstract}
An important class of three-manifolds are L-spaces, which are rational homology spheres with the smallest possible Floer homology.  For knots with an instanton L-space surgery, we compute the framed instanton Floer homology of all integral surgeries.  As a consequence, if a knot has a Heegaard Floer and instanton Floer L-space surgery, then the theories agree for all integral surgeries.  In order to prove the main result, we prove that the Baldwin-Sivek contact invariant in framed instanton Floer homology is homogeneous with respect to the absolute $\Z/2$-grading, but not the $\Z/4$-grading.
\end{abstract}
\maketitle

%-------------------------------------------------------------------------------------------------
%-------------------------------------------------------------------------------------------------
\section{Introduction}
Over the past 30 years, beginning with the pioneering work of Floer, several variants of Floer homology for three-manifolds have been constructed having a powerful impact on topology.  In particular, Floer homology has been very useful for questions in Dehn surgery, including solving the Property P \cite{km-propertyp} and the lens space realization conjectures \cite{greene}, using instanton and Heegaard Floer homology respectively.  While many technical tools have been developed for studying these invariants, one major gap still seems to exist in the theory, namely connecting instanton Floer homology with other isomorphic Floer homologies: monopole Floer homology, Heegaard Floer homology, and embedded contact homology.  These latter three are all known to be isomorphic by work of Kutluhan--Lee--Taubes \cite{klt1, klt2, klt3, klt4, klt5}, Taubes \cite{taubes-1, taubes-2, taubes-3, taubes-4, taubes-5} and Colin--Ghiggini--Honda \cite{colin-ghiggini-honda}; additionally, these three each have a very distinct nature that can be combined in a powerful way, such as in Taubes's proof of the Weinstein conjecture \cite{taubes-weinstein}.  However, instanton Floer homology is the one invariant which is rooted in the fundamental group, and thus its connection to the other theories is crucial.  It is important to note that much of the power of all four of these theories comes, sometimes indirectly, from a connection with contact topology.  

In the current article, we relate framed instanton Floer homology and Heegaard Floer hat homology for surgeries on certain knots in $S^3$, and in particular give an explicit formula for the $\Z/2$-graded instanton Floer homology of integral surgery along an instanton L-space knot.  Along the way, we put an absolute $\Z/2$-grading on the Baldwin--Sivek contact invariant in instanton homology (see \cite{baldwin-sivek}).    

For background, we will say that a 3--manifold $Y$ with the same rational homology as $S^3$ is a {\em Heegaard Floer L-space} if $\dim \HFhat(Y;\mathbb{Q}) = |H_1(Y;\Z)|$.\footnote{Experts may notice that this varies slightly from certain conventions for Heegaard Floer L-space. As there is no known torsion in the Heegaard Floer homology of rational homology spheres, all possible definitions of L-space may agree.}  Additionally, a knot $K$ in $S^3$ is a {\em Heegaard Floer L-space knot} if there exists a positive surgery slope $r\in\Q$ such that $S^3_r(K)$, the result of $r$-surgery on $S^3$ along $K$, is a Heegaard Floer L-space. Examples of these (after possibly mirroring) are torus knots, Berge knots, and $(1,1)$-knots with coherent diagrams \cite{greene-lewallen-vafaee}.  It is well-known that if $K$ is a non-trivial Heegaard Floer L-space knot, then $S^3_r(K)$ is a Heegaard Floer L-space if and only if $r \geq 2g(K) - 1$. Here, and throughout, $g(K)$ denotes the Seifert genus.  Heegaard Floer L-space knots also are known to be strongly quasipositive \cite{hedden} and fibered \cite{ghiggini, ni}, and hence maximize the transverse Bennequin inequality, that is  $\maxsl(K) = 2g(K) - 1$, where $\maxsl$ denotes the maximal self-linking number among all transverse representatives.  Note that for such knots the Seifert genus agrees with the slice genus.

In analogy, a rational homology sphere is an {\em instanton L-space} if the framed instanton homology $I^\#(Y;\mathbb{C})$ satisfies $\dim \Isharp(Y;\mathbb{C}) = |H_1(Y;\mathbb{Z})|$. Many similar results have been established for {\em instanton L-space knots}, including being fibered and strongly quasipositive \cite{baldwin-sivek-L-space}.  It is not known if these two notions of L-spaces agree, but they do line up on many examples, such as lens spaces.  To keep notation simple, all Floer homologies will be computed with $\C$-coefficients (even Heegaard Floer) and singular homology will be assumed to have $\mathbb{Z}$-coefficients.  Unless specified otherwise, instanton Floer homology refers to $\Isharp$ and Heegaard Floer homology refers to $\HFhat$.    

Our first result is a formula for the instanton Floer homology of integral surgery on instanton L-space knots.  
\begin{theorem}\label{dim-Isharp}
Let $K$ be a non-trivial knot in $S^3$ with a positive instanton L-space surgery.  Then, setting $g=g(K)$, as absolutely $\Z/2$-graded $\C$-vector spaces, we have 
\begin{equation}\label{eq:dim-Isharp}
\Isharp(S^3_n(K)) \cong \begin{cases} \C^{2g -1 - n}_{(0)} \oplus \C^{2g - 1}_{(1)} &  n \leq 0 \\
\C^{2g -1}_{(0)} \oplus \C^{2g - 1 - n}_{(1)} &   0 \leq n \leq 2g-1 \\
\C_{(0)}^n & n \geq 2g - 1.
\end{cases}
\end{equation}
\end{theorem}
A good example of a knot satisfying the conditions of \Cref{dim-Isharp} is a positive torus knot, since a lens space is an instanton L-space and positive torus knots maximize the transverse Bennequin inequality.  The instanton Floer homology groups of most surgeries on torus knots were previously unknown.  

%\begin{remark}Recent work of Baldwin and Sivek \cite{baldwin-sivek-L-space} shows that a knot $K$ in $S^3$ with a positive instanton L-space surgery is fibered and strongly quasipositive, and therefore immediately satisfies $\maxsl(K) = 2g(K) - 1$.
%\end{remark}

Since \Cref{dim-Isharp} is also true for Heegaard Floer L-space knots, even with $\C$-coefficients, we have the following corollary: 

\begin{corollary}\label{IsharpisHFhat}
Let $K$ be a knot in $S^3$ which is both an instanton and Heegaard Floer L-space knot.  Then, for all $n \in \mathbb{Z}$, there is an isomorphism of $\Z/2$-graded $\C$-vector spaces:
\[
\HFhat(S^3_n(K)) \cong \Isharp(S^3_n(K))
\]
Furthermore, the isomorphism type is given by \eqref{eq:dim-Isharp}.
\end{corollary}

In the case that an additional grading constraint holds on one instanton L-space surgery on $K$, such as in the case of a knot with a positive lens space surgery, then we can in fact completely compute $I^\#(S^3_n(K))$ as a $\Z/4$-graded group.  The precise statement can be found in \Cref{Z4-grading}.  While our arguments do not seem to apply more generally for rational surgery, they can be used to recover $\Isharp$ for $1/n$-surgeries on the right-handed trefoil with $n > 0$ in \Cref{sec:trefoil}.   

A rough sketch of the proof of \Cref{dim-Isharp} is as follows.  For $m \gg 0$, $Y = S^3_{-m}(K)$ admits several Stein fillings with distinct canonical classes.  By work of Baldwin--Sivek \cite{baldwin-sivek-su2}, the instanton contact invariants for the associated contact structures on $Y$ in $\Isharp(-Y)$ are linearly independent, and this gives a lower bound on $\Isharp(-Y)$ and hence $\Isharp(Y)$, as well as all other surgeries by the surgery exact triangle.  An upper bound can be obtained by a standard application of the surgery exact triangle together with the existence of an instanton L-space surgery.  In order to obtain the desired lower bound, it is necessary to exhibit an absolute $\Z/2$-grading on the Baldwin--Sivek contact invariant:  

\begin{theorem}\label{contact-element}
	Suppose $\xi$ is a contact structure on a closed oriented 3-manifold $Y$. Then the Baldwin--Sivek contact invariant $\Theta^\#(\xi)\in I^\#(-Y)$ is homogeneous with respect to the absolute $\Z/2$-grading.  Furthermore, this grading can be computed explicitly in terms of the algebraic topology of an almost-complex four-manifold bounding $(Y,\xi)$.    
\end{theorem}

A corollary of the grading computation of \Cref{contact-element} is that the mod 2 grading of the instanton contact invariant agrees with that of the Heegaard Floer contact invariant for rational homology 3-spheres.

\begin{corollary}\label{contact-gradings-agree}
Let $\xi$ be a contact structure on a rational homology 3-sphere $Y$.  Then the absolute $\Z/2$-gradings of $c(\xi) \in \HFhat(-Y)$ and $\Theta^\#(\xi) \in \Isharp(-Y)$ agree.  
\end{corollary}

If $c_1(\xi)$ is torsion, then $c(\xi)$ inherits an absolute {\em rational} grading, related to the Gompf invariant $d_3(\xi)$.  While there is no rational grading currently defined on instanton Floer homology, there is an absolute $\Z/4$-grading.  Unfortunately, the contact invariant is not homogeneous with respect to this grading.

\begin{proposition}\label{contact-gradings-Z4}
Let $Y = \Sigma(2,3,7)$ oriented as the boundary of a negative-definite plumbing.  Then, there exists a Stein fillable contact structure $\xi$ on $Y$ such that $\Theta^\#(\xi)$ is not homogeneous with respect to the absolute $\Z/4$-grading.  
\end{proposition}    

\subsection*{Organization} In \Cref{sec:instantons}, we give the necessary background on framed instanton Floer homology.  We recall the construction of the contact invariant in \Cref{sec:contact} and prove \Cref{contact-element}.  The proof of \Cref{dim-Isharp} is given in \Cref{sec:dim-Isharp}.  The generalization to $\Z/4$-gradings is discussed in \Cref{sec:mod4}.  Finally, the computation of $\Isharp$ for $1/n$-surgeries is done in \Cref{sec:trefoil}.  The reader solely interested in the proof of \Cref{dim-Isharp} without the details of the $\Z/2$-gradings may skip directly to \Cref{sec:dim-Isharp}. 

\subsection*{Acknowledgements} We thank John Baldwin for pointing out an improvement of \Cref{lower-bound} from $\maxtb$ to $\maxsl$, and both John Baldwin and Steven Sivek for several very useful and entertaining discussions.  We also thank Ko Honda, Laura Starkston and Jeremy Van Horn-Morris for some very helpful remarks.

%-------------------------------------------------------------------------------------------------
%-------------------------------------------------------------------------------------------------
\section{Background on Instanton Homology}\label{sec:instantons}
In this section we introduce the basic definitions of instanton Floer homology, with a special emphasis on its gradings and the degree of the maps induced by cobordisms. This is because the proof of our main theorem relies on a careful examination of the grading of the elements of the instanton homology determined by contact structures on the given 3--manifold.

\subsection{Floer's instanton homology for admissible pairs}

\begin{definition}\label{admissible} A pair $(Y,\omega)$ is said to be  \emph{admissible} if $Y$ is a connected, closed, oriented 3--manifold and $\omega\subset Y$ is an unoriented, closed, 1--dimensional submanifold of $Y$ that has odd intersection with some oriented and closed surface in $Y$. \end{definition}

Instanton homology was introduced by Floer in ~\cite{floer-invt,floer-dehn,floer-dehn-old}. Here we restrict our attention to the version defined for admissible pairs $(Y,\omega)$. Let $E$ be an $SO(3)$-bundle over $Y$ with $w_2(E)$ Poincar\'{e} dual to $[\omega]\in H_1(Y;\Z/2)$. The instanton homology $I(Y)_\omega$ for the admissible pair $(Y,\omega)$ is morally the Morse homology of the circle-valued Chern--Simons functional defined on the space of connections on $E$ modulo ``restricted'' gauge transformations that lift to $SU(2)$ in a suitable sense, see for example  \cite[Part II]{braam-donaldson}. The admissibility condition guarantees that there are no reducible flat connections on $E$. We will use only complex coefficients, in which case $I(Y)_\omega$ is a vector space over $\C$ with a $\Z/2$-grading, which lifts to a relative $\Z/8$-grading. We note that $I(Y)_\omega$ only depends on $Y$ and $[\omega]\in H_1(Y;\Z/2)$.

In ideal cases, the generators for the chain complex of $I(Y)_\omega$ are flat connections mod gauge, and gradient flow lines correspond to finite energy instantons on $Y\times \R$, but in practice perturbations must be employed to achieve transversality. The $\Z/2$-grading of a (perturbed) flat connection $a$ in the chain complex is given by\footnote{This $\Z/2$-grading differs from that defined in \cite{donaldson-book} by a shift of $b_1(Y) \pmod 2$.}
\[
	\text{gr}(a) := -\text{ind}(D_A) - b_1(X)+b_+(X)-b_1(Y)\pmod 2.
\]
Here $X$ is an oriented 4-manifold with boundary $Y$, and $A$ is a connection extending $a$ on an $SO(3)$-bundle over $X$ that restricts to $E$ over $Y$. The operator $D_A=d_A^\ast\oplus d_A^+$ is the ASD operator associated to $A$, and its index equals the expected dimension of the moduli space of instantons homotopic rel $a$ to the gauge equivalence class of $A$.

Given an oriented closed submanifold $S\subset Y$, there is an associated endomorphism $\mu(S)$ of $I(Y)_\omega$ which has $\Z/8$-degree $\dim S-4$, and only depends on the homology class of $S$ in $Y$. The endomorphisms for even-dimensional submanifolds commute, and so they have simultaneous eigenspaces. For example, for any point $y\in Y$ and any surface $R\subset Y$ the endomorphisms $\mu(y)$ and $\mu(R)$ commute and have $\Z/8$-degrees $-4$ and $-2$ (resp.), and so both have degree $0 \pmod 2$. In the case when the intersection $\omega\cdot R$ is odd and $R$ has positive genus, using work of Mu\~{n}oz~\cite{munoz}, Kronheimer and Mrowka~\cite[Corollary 7.2]{km-excision} show that the eigenvalues of $(\mu(R),\mu(y))$ that correspond to their simultaneous eigenspaces are included in the set $$ \{(i^r(2k), (-1)^r2) \mid k\in\Z,\; 0\leq k\leq g-1,\,r\in\{0,1,2,3\}\}.$$ Following \cite[Section 7]{km-excision}, we will restrict our attention to the eigenspace of $I(Y)_\omega$ that corresponds to the eigenvalues $(2g-2,2)$.

\begin{definition} Let $(Y,\omega)$ be an admissible pair, $y\in Y$ a basepoint, and $R\subset Y$ a closed oriented surface that has odd intersection with $\omega$. Suppose $R$ has genus $g\geq 1$. Define $$I(Y|R)_\omega\subset I(Y)_\omega$$  to be the simultaneous eigenspace for the operators $(\mu(R),\mu(y))$ corresponding to the eigenvalues $(2g-2,2)$. \end{definition}

\begin{remark}\label{rem::torus} As the endomorphisms $\mu(R)$ and $\mu(y)$ have degree $0 \pmod 2$, the vector space $I(Y | R)_\omega$ inherits a $\Z/2$-grading from $I(Y)_\omega$. Furthermore, note that if $R$ has genus 1, the endomorphism $\mu(R)$ has zero as its unique eigenvalue, and thus $I(Y | R)_\omega$ is simply the $+2$-eigenspace of $\mu(y)$ acting on $I(Y)_\omega$. As $\mu(y)$ has degree $-4$ with respect to the relative $\Z/8$-grading, in this case $I(Y |  R)_\omega$ inherits a relative $\Z/4$-grading.
\end{remark}

Next, for an orientable, connected 4-dimensional cobordism $W:Y_1\to Y_2$ we define the integer
\begin{equation}
	d(W) := -\frac{3}{2}(\chi(W)+\sigma(W))+\frac{1}{2}(b_1(Y_2)-b_1(Y_1) + b_0(Y_2)- b_0(Y_1)). \label{eq:ddeg}
\end{equation}
Note that $d(W)$ generally depends on the cobordism structure of $W$, as opposed to the underlying 4-manifold.  In other words, $d(W)$ depends additionally on which components of the boundary are considered incoming or outgoing. Suppose we have an oriented cobordism $(W,\nu):(Y_1,\omega_1)\to (Y_2,\omega_2)$ between admissible pairs, where $W$ is connected and $\nu: \omega_1 \to \omega_2$ is an oriented cobordism embedded in $W$. There is an associated linear map
\begin{equation}\label{cob-admissible}
I(W)_\nu:I(Y_1)_{\omega_1}\to I(Y_2)_{\omega_2}.
\end{equation}
Note that the surface $\nu$ is oriented, even though the curves $\omega_1$ and $\omega_2$ are not; the Poincar\'{e} dual of $[\nu]$ in $H^2(W;\Z)$ corresponds to the lift of $w_2$ of an $SO(3)$-bundle over $W$, listed as the second item in \cite[Part II, 1.2]{braam-donaldson}. 

On occasion we consider $(W,\nu):(Y_1,\omega_1)\to (Y_2,\omega_2)$ where one of $(Y_i,\omega_i)$ is empty, or is a disjoint union of admissible pairs, and the construction of \eqref{cob-admissible} extends to these cases with the following remarks. First, the instanton homology of a disjoint union of admissible pairs is the tensor product of the instanton homologies of the connected component admissible pairs. Second, if $(Y_1,\omega_1)$ is empty, then $I(W)_\nu$ is defined to be an element of $I(Y_2)_{\omega_2}$. Third, if $Y_1$ is disconnected, then the composition law for cobordism maps generally only holds up to a nonzero constant, see for example  \cite[Section 3.2]{braam-donaldson}. However, this latter point will not be important in the sequel. (It may be corrected by carefully choosing gauge groups, as explained in \cite[Section 5.2]{km-unknot}.) 

The $\Z/2$-degree of the cobordism-induced map \eqref{cob-admissible} is given by 
\begin{equation}
	d(W) \equiv \frac{1}{2}\left(\chi(W)+\sigma(W)+b_1(Y_2)-b_1(Y_1)+ b_0(Y_2)- b_0(Y_1)\right) \pmod 2\label{eq:dmod2}
\end{equation}
and so in particular, it does not depend on $\nu$. Furthermore, the map $I(W)_\nu$ only depends on the cobordism $W$ and the homology class of $\nu$.

In the case where  $R_1$ and $R_2$ are embedded surfaces in $Y_1$ and $Y_2$ as above, which are each homologous to $S$, an embedded surface in the cobordism $W$, there is a map $$I(W | S)_{\nu}:I(Y_1 | R_1)_{\omega_1}\to I(Y_2 | R_2)_{\omega_2}.$$ The $\Z/2$-degree of this map is also given by $d(W) \pmod 2$, as this map is simply the restriction of the map from \eqref{cob-admissible} to appropriate eigenspaces.

It will be useful to understand the instanton homology of $\Sigma_g \times S^1$.  The following is from \cite[Propositions 7.8, 7.9]{km-excision}, which uses \cite{munoz} and excision.

\begin{proposition}\label{prop:1diminst}
Consider $\Sigma_g\times S^1$ where $\Sigma_g$ is a closed oriented surface of genus $g$. Let $\omega \subset \Sigma_g\times S^1$ be either (i) an $S^1$ factor, (ii) a primitive closed curve on $\Sigma_g$, or the disjoint union of two such curves in (i) and (ii). Then $I(\Sigma_g\times S^1 | \Sigma_g)_\omega$ is 1-dimensional.
\end{proposition}

In case (i), this 1-dimensional vector space is supported in grading $0 \pmod 2$:

\begin{lemma}\label{lemma:1dim0mod2}
$I(\Sigma_g\times S^1 | \Sigma_g)_{S^1} \cong \C_{(0)}$.
\end{lemma}

\begin{proof}
	The 1-dimensional space $I(\Sigma_g\times S^1 | \Sigma_g)_{S^1}$ is contained in the subspace of $I(\Sigma_g\times S^1)_{S^1}$ invariant under the action of $\text{Sp}(2g,\Z)$ induced from the mapping class group action of $\Sigma_g$. This follows from \cite{munoz}, where it is shown that this invariant subspace is generated by
\begin{equation}
		\mu(\Sigma_g)^a \mu(y)^b \left( \textstyle{\sum_{j=1}^{g} }\mu(\gamma_j)\mu(\gamma_{j+g})\right)^c I(\Sigma_g\times D^2)_{D^2\cup i\Sigma_g} \qquad (a,b,c)\in \Z_{\geq 0}^3, \; i\in \{0,1\}\label{eq:munozgens}
\end{equation}
where $\gamma_1,\ldots,\gamma_{2g}$ form a basis of curves in $\Sigma_g$ with $\gamma_i\cdot \gamma_{i+g}=1$ for $1 \leq i \leq g$ and $\gamma_i\cdot \gamma_{j}=0$ otherwise. The endomorphism in front of $I(\Sigma_g\times D^2)_{D^2\cup i \Sigma_g}$ in \eqref{eq:munozgens} is always of degree $0 \pmod 2$, so it suffices to show that the grading of $I(\Sigma_g\times D^2)_{D^2\cup i\Sigma_g}$ is $0 \pmod 2$. For this, use \eqref{eq:dmod2}:
\begin{equation*}
	d(\Sigma_g\times D^2) \equiv \frac{1}{2}\left( \chi(\Sigma_g) + 0  + (2g+1)-0 + 1-0\right) \equiv 0 \pmod 2. \qedhere
\end{equation*}
\end{proof}

In the next section we will see that all the 1-dimensional instanton homology groups of \Cref{prop:1diminst} are supported in grading $0 \pmod 2$.

\begin{example}\label{ex:S3} We say more about the special case of the admissible pair $(T^3, S^1)$, the description of which is justified by the more general work of Mu\~noz \cite{munoz}. First, $\smash{I(T^3)_{S^1}\cong \C_{(0)}^2}$. Indeed, no perturbation is needed for the critical set of the Chern--Simons functional, as for a non-trivial $SO(3)$-bundle over $T^3$, there are exactly two non-degenerate flat connections modulo restricted gauge transformations (and only one if we use the full $SO(3)$ gauge group). These generators have relative $\Z/8$-grading difference $4 \pmod 8$, so there is no differential. The two generators also correspond to the relative invariants $I(T^2\times D^2)_{D^2}$ and $\mu(y)I(T^2\times D^2)_{D^2}$. In this case $\mu(y)^2=4$, and we recover
\[
 I(T^3 |  T^2)_{S^1} = \text{span}\left\langle I(T^2\times D^2)_{D^2} + \frac{1}{2}\mu(y)I(T^2\times D^2)_{D^2} \right\rangle \cong \C_{(0)}.
 \]
\end{example}

\subsection{Framed instanton homology}

Framed instanton homology is a special case of the construction $I(Y | R)_\omega$ for an admissible pair. To define it, given any closed, connected, oriented 3--manifold $Y$ choose a basepoint $y\in Y$. Let $T^3$ have distinguished basepoint $p^\#$. We may then form the connected sum of $Y$ and the 3-torus $T^3= T^2\times S^1$ with respect to these basepoints. We always view $T^3$ as the product $T^2\times S^1$. 

\begin{definition}The framed instanton homology of a closed, connected, oriented, based 3--manifold $Y$ is defined to be the $\Z/2$-graded complex vector space\footnote{This is technically different from the definition in \cite{scaduto}, but gives an isomorphic theory.}
\[ I^\#(Y) := I(Y\# T^3 | T^2)_{S^1}.\]
\end{definition}

We remark that $I^\#(-Y)$ is $\Z/2$-graded isomorphic to the dual of $I^\#(Y)$ with the grading shifting from $i$ to  $b_1(Y)-i$.

Notice that by \Cref{rem::torus} the vector space $I^\#(Y)$ inherits a relative $\Z/4$-grading. This in fact lifts to an absolute $\Z/4$-grading, which we return to in \Cref{sec:mod4}. The previous duality statement also holds for this $\Z/4$-grading.

\begin{example}
By \Cref{ex:S3}, we see that $I^\#(S^3)$ is isomorphic to $\C_{(0)}$.  
\end{example} 

Let $W:Y_1\to Y_2$ be a 4-dimensional cobordism, and let $\gamma$ be a path in $W$ that joins the basepoints in $Y_1$ and $Y_2$. Then we have an induced map
\begin{equation}\label{cob-sharp}
	I^\#(W):I^\#(Y_1)\to I^\#(Y_2)
\end{equation}
whose degree is $d(W)$.  We will omit the path $\gamma$ from the notation. More generally, for an oriented closed surface $\nu \subset W$ we obtain a similar map $I^\#(W)_\nu$ only depending on $W$ and $[\nu]\in H_2(W;\Z)$. (There is also the case in which there are curves in $Y_1$ and $Y_2$ and $(W,\nu)$ is a cobordism of pairs; however, this will not appear in the sequel.) The degree of the map \eqref{cob-sharp} is also given by $d(W) \pmod 2$, written in \eqref{eq:dmod2}; see \cite[Proposition 7.1]{scaduto}.

\subsection{Exact triangle in framed instanton homology}
Given a knot $K$ in $Y$, there exist several exact triangles relating the framed instanton homology of $Y$ and various Dehn surgeries on $K$. These are all special cases of Floer's original exact triangle \cite{floer-dehn-old, floer-dehn}.  We focus on the particular case of knots in integer homology spheres.  (For more details, see \cite[Section 7.5]{scaduto}.) 

\begin{theorem}\label{exact-triangle} 
Let $K$ be a knot in an integer homology sphere $Y$.  For any $n \in \Z$, there is an exact triangle
\begin{equation}\label{eq:ntriangle}\xymatrix{I^\#(Y_n(K))\ar[rr]&&I^\#(Y_{n+1}(K))\ar[dl] \\
&I^\#(Y)\ar[ul]&
}
\end{equation}
\end{theorem}
The cobordism maps are induced by the corresponding 2-handle attachments with certain embedded surfaces $\nu$, and the mod 2 degrees are easily computed from \eqref{eq:ddeg}. For a discussion regarding the mod 4 gradings in the exact triangle, see \Cref{sec:mod4}.

%-------------------------------------------------------------------------------------------------
%-------------------------------------------------------------------------------------------------
\section{The Instanton Contact Invariant}\label{sec:contact}
We now define Baldwin and Sivek's instanton contact invariant for a contact 3-manifold, originally constructed in \cite{baldwin-sivek}. We mostly follow the description given in \cite{baldwin-sivek-su2}, adding some slight modifications that allow us to work in the context of framed instanton homology.

\subsection{Review of the contact invariant}\label{sec::contactinvt}
An {\emph{open book}} is a triple $(S,h,\mathbf{c})$ in which $S$ is a connected, compact, oriented surface with nonempty boundary; $h$ is a diffeomorphism of $S$ with $h|_{\partial S}$ the identity; and $\mathbf{c}=\{c_1, . . . , c_{b_1(S)}\}$ is a set of disjoint, properly embedded arcs with the property that $S\setminus \mathbf{c}$ deformation retracts onto a point. Define the curves $\gamma_1,\ldots,\gamma_{b_1(S)}$ in $\partial \left(S \times [-1,1] \right)$ as:
 \begin{equation}\label{gamma}
 	\gamma_i = \left(c_i\times \{1\}\right)\cup \left( \partial c_i \times [-1,1] \right) \cup \left( h(c_i)\times \{-1\} \right) \subset \partial \left(S\times [-1,1]\right).
 \end{equation}
A contact 3-manifold $M(S,h,\mathbf{c})$ is built from this data by starting with a standard contact structure on $S\times [-1,1]$ (with rounded corners), and attaching $b_1(S)$ contact 2-handles along the $\gamma_i$ curves. This contact 3-manifold has a 2-sphere boundary.

Let $(Y,\xi,p)$ be a based contact 3-manifold. For us this means that $Y$ is a closed, oriented and connected 3-manifold, $\xi$ is a co-oriented contact 2-plane field on $Y$, and $p\in Y$ is a basepoint. An {\emph{open book decomposition}} of $(Y,\xi,p)$ is an open book $(S,h,\mathbf{c})$ together with a contactomorphism $f:M(S,h,\mathbf{c})\to Y\setminus B(p)$, where $B(p)$ is a Darboux ball containing $p$. It follows from the work of Giroux \cite{giroux} that every based contact 3-manifold has an open book decomposition. It is also possible to choose an open book decomposition with $\partial S$ connected. We fix such a choice $(S,h,\mathbf{c},f)$ for $(Y,\xi,p)$.

Next, choose a connected, compact, oriented surface $T$ with an identification $\partial T \cong -\partial S$, and $g(T)\geq 2$\footnote{Although Baldwin and Sivek \cite{baldwin-sivek-su2} assume this genus to be at least 8, it turns out this is not necessary.}. Let $R= S\cup T$ be the closed surface obtained from gluing $S$ and $T$. We extend $h$ to a diffeomorphism of $R$, which we also call $h$, by letting it act on $T\subset R$ as the identity. The mapping torus $R\times_h S^1$ will be viewed as the quotient
\[
	R \times [-1,3] / \left((x,3)\simeq (h(x),-1)\right).
\]
Each $\gamma_i$ defined above may then be viewed as an embedded curve in $\partial \left( S\times [-1,1]\right) \subset R\times_h S^1$. Let $V$ be the cobordism obtained from $(R\times_h S^1)\times [0,1]$ by attaching $b_1(S)$ many 4-dimensional 2-handles to $(R\times_h S^1)\times \{1\}$ along the curves $\gamma_i$ with framing induced by $\partial \left( S\times [-1,1]\right)\times \{1\}$. The outoing end of $V$ is diffeomorphic to $Y\# R\times S^1$ by \cite[Proposition 2.16]{baldwin-sivek-su2}, so we may write this cobordism as follows:
\[
	V: R\times_h S^1 \longrightarrow Y\# R\times S^1.
\]
Let $\eta\subset  \text{int}(T)\times \{2\}\subset R\times_h S^1$ be an oriented, embedded, nonseparating curve, and let $\alpha$ denote the closed loop $\{q\} \times [-1,3]\subset R\times_h S^1$, where $q\in \text{int}(T)$ is disjoint from $\eta$. We may also view these same curves as embedded in $R\times S^1$. The data $(Y\# R\times S^1,R,\alpha,\eta)$ naturally determine, in the terminology of Definition 2.12 of \cite{baldwin-sivek}, a {\emph{marked odd closure}} of $Y\setminus B(p)$, which we denote by $\mathcal{D}_R$. Let $\nu = (\alpha \sqcup \eta)\times [0,1]\subset V$. Then consider
\[
	I(-V|-R)_\nu: I(-R\times_h S^1|-R)_{\alpha \sqcup \eta} \longrightarrow I(-Y\# R\times S^1|-R)_{\alpha \sqcup \eta}.
\]
By \Cref{prop:1diminst} and excision (see \cite[Section 7.3]{km-excision}), the domain of this map is 1-dimensional. Let $1\in I(-R\times_h S^1|-R)_{\alpha \sqcup \eta}$ be any generator. Then
\begin{equation}
	\theta(\mathcal{D}_R) := I(-V|-R)_\nu(1)\in I(-Y\# R\times S^1|-R)_{\alpha \sqcup \eta}\label{eq:thetadef}
\end{equation}
is the contact invariant considered in \cite{baldwin-sivek-su2}. This is defined up to multiplication by $\C^\times$. Thus the contact invariant is really a subspace, of dimension 0 or 1. We would like to view this invariant as a subspace of $I^\#(-Y)$, and to do this we proceed as follows. The constructions go back to \cite{km-excision}, and in fact the most important step, for our purposes, involves genus-decreasing excision cobordism maps, see for example  \cite[Figure 5]{km-excision}. We will also see that what follows fits into the framework of \cite{baldwin-sivek-nat, baldwin-sivek}.

We first fix, for each positive integer $g$, a closed oriented surface $\Sigma_g$ of genus $g$. On $\Sigma_g$ we fix, as depicted in \Cref{fig:excision}, non-separating embedded closed curves $\eta_g$, $a_i^g$, $b_i^g$ ($i=1,\ldots,g$), as well as basepoints $p_g$ and $q_g$ away from the chosen curves. Note that the $a_i^g$ and $b_i^g$ are all pairwise disjoint, and they each intersect $\eta_g$ transversely in exactly one point.

 \begin{figure}
\centering
 \includegraphics[scale=1]{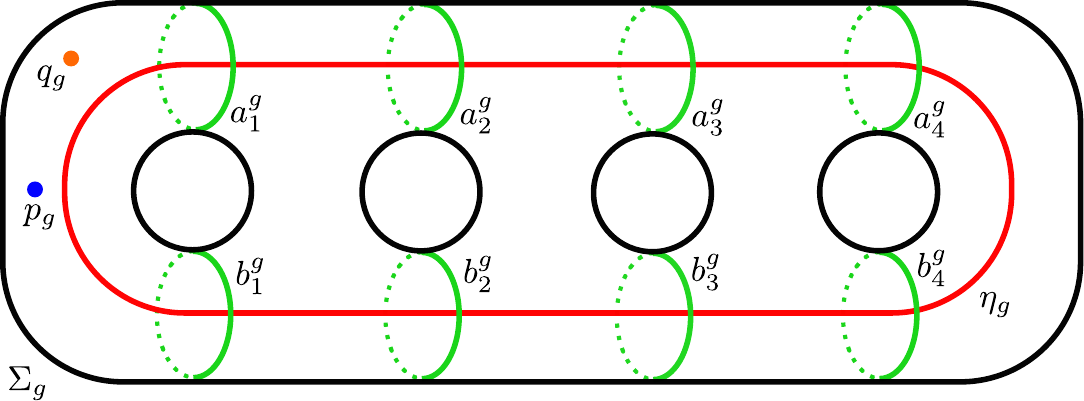}
\caption{The surface $\Sigma_g$ with its curves and basepoints for $g=4$.}\label{fig:excision}
\end{figure}

Next, choose a diffeomorphism $\phi:R\to \Sigma_g$ such that $\phi(\eta)=\eta_g$, $\phi(q)=q_g$, and such that the point in $S\times \{0\}\subset R\times S^1$ at which the connected sum with $Y$ is performed is sent under $\phi$ to $p_g$. Define $\alpha_g = \{q_g\} \times S^1 \subset \Sigma_g \times S^1$. We have an induced isomorphism
\[
	 F_{\phi}:I(-Y\# R\times S^1|-R)_{\alpha \sqcup \eta} \longrightarrow I(-Y\# \Sigma_g\times S^1|-\Sigma_g)_{\alpha_g \sqcup \eta_g}.
\]
Suppose $g\geq 2$. Consider the two closed curves $C^g_1=a^g_{g-1}$, $C_2^g=b^g_{g-1}$ in $\Sigma_g$. Then $\Sigma_g\setminus (C^g_1\sqcup C^g_2)$ is a union of two components, $\Sigma_g'$ and $\Sigma_g''$, such that:
\begin{itemize}
	\item $\Sigma_g''$ is a genus 1 surface with two boundary components,
	\item the point $q_g$ defining the curve $\alpha_g$ lies in $\Sigma_g'$,
	\item the point $p_g \times \{0\} \in \Sigma_g\times S^1$ where the connect sum is taken with $Y$ is in $\Sigma_g'\times S^1$.
\end{itemize}
Each of $T_i=C^g_i\times [-1,3]\subset \Sigma_g\times S^1$ may be viewed as a closed genus 1 surface in $Y\#\Sigma_g\times S^1$. Cut open $Y\#\Sigma_g\times S^1$ along $T_1$ and $T_2$. The resulting boundary of the cut open manifold has four components: $T_1', T_2' \subset \Sigma_g'$ and $T_1'',T_2'' \subset \Sigma_g''$. Gluing $T_1'$ to $T_2'$, and $T_1''$ to $T_2''$, yields
 \begin{figure}
\centering
 \includegraphics[scale=.7]{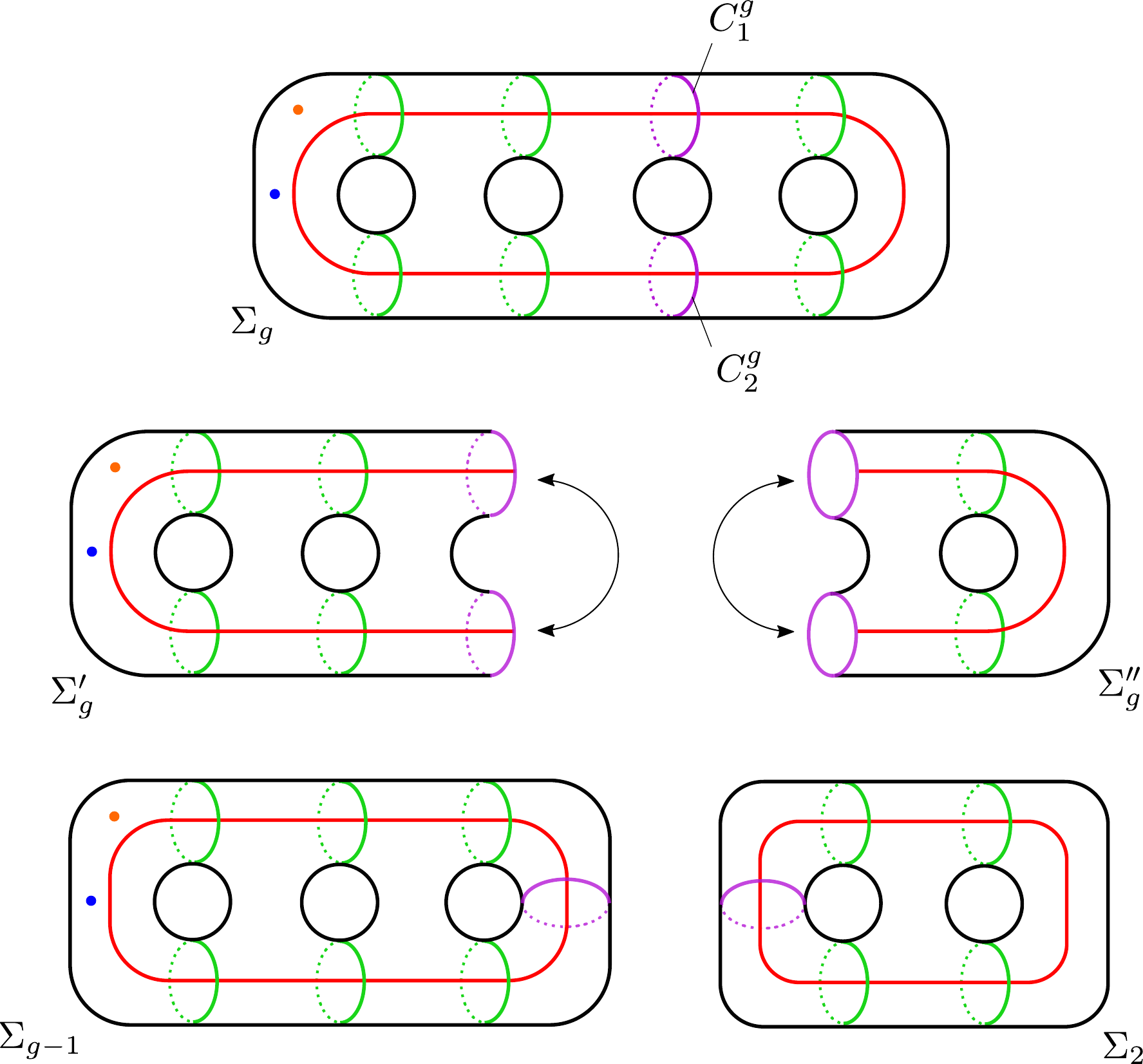}
\caption{Cutting open $\Sigma_g$ and regluing to obtain $\Sigma_{g-1}$ and $\Sigma_2$.}\label{fig:cut}
\end{figure}
 \[
 	Y\# \Sigma_{g-1}\times S^1 \sqcup \Sigma_2 \times S^1
 \]
where we have identified the result of gluing the boundary components of $\Sigma_g'$ with $\Sigma_{g-1}$ and the result of gluing the boundary components of $\Sigma_g''$ with a closed genus 2 surface $\Sigma_2$. This cutting and regluing is done in an $S^1$-equivariant fashion, and is illustrated on $\Sigma_g\times \{0\}\subset Y\# \Sigma_g\times S^1$ in \Cref{fig:cut}. Note that the result of gluing the boundary components of $\Sigma_g'$ is naturally identified with $\Sigma_{g-1}$ in such a way that the part of $\eta_g$ which is reglued is identified with $\eta_{g-1}$, each $a_i^g$ and $b_i^g$ for $i=1,\ldots,g-2$ is identified with $\smash{a_i^{g-1}}$ and $\smash{b_i^{g-1}}$ in $\Sigma_{g-1}$, and the basepoints $p_g$ and $q_g$ go to $p_{g-1}$ and $q_{g-1}$, respectively. There is a cobordism
 \[
 	W^g:Y\# \Sigma_g\times S^1 \longrightarrow \left(Y\# \Sigma_{g-1}\times S^1\right) \sqcup  \left(\Sigma_2 \times S^1\right)
\]
obtained by gluing the following three pieces, where $P$ is a 2-dimensional saddle:
\begin{align}
	W_1^{g} & =  [0,1]\times (Y\# \Sigma_g\times S^1 \setminus \Sigma_g''\times S^1) \label{eq:w1}\\
	W_2^{g} & = P \times T_1 \label{eq:w2}\\
	W_3^{g} & = [0,1]\times (\Sigma_g''\times S^1). \label{eq:w3}
\end{align}
See for example  \cite[\S 3.2]{km-excision} and \cite[Figure 11]{baldwin-sivek-nat} for more details. We obtain a map
\[
	I(-W^g)_\nu:I(-Y\# \Sigma_g\times S^1)_{\alpha_g\sqcup \eta_g} \longrightarrow I(-Y\# \Sigma_{g-1}\times S^1)_{ \alpha_{g-1}\sqcup \eta_{g-1}} \otimes  I(-\Sigma_2 \times S^1)_{\eta_2}
\]
induced by $W^g$ and a surface cobordism $\nu\subset W^g$ prescribing a bundle, the construction of which is straightforward, but whose details will not concern us. The excision theorem \cite[Theorem 7.7]{km-excision} implies that $I(-W^g)_\nu$ is an isomorphism when restricted to the $+2$ eigenspaces of the point class operators; notice that the operators $\mu(T_i)$ vanish because $g(T_i)=1$. Further, $I(-W^g)_\nu$ intertwines the operator $\mu(-\Sigma_g)$ on the domain with $\mu(-\Sigma_{g-1})\otimes 1 + 1\otimes \mu(-\Sigma_2)$ on the codomain, just as in \cite[Proposition 7.9]{km-excision}. Furthermore, the Floer group $I(-\Sigma_2 \times S^1|-{\Sigma_2})_{\eta_2}$ is 1-dimensional, by \Cref{prop:1diminst}, and upon identifying it with $\C$, we obtain from the restriction of $I(-W^g)_\nu$ an isomorphism
\[
	\Psi^g:I(-Y\# \Sigma_g\times S^1|-\Sigma_g)_{\alpha_g\sqcup \eta_g} \longrightarrow I(-Y\# \Sigma_{g-1}\times S^1|-\Sigma_{g-1})_{\alpha_{g-1} \sqcup\eta_{g-1}}.
\]
The codomain of $\Psi^2$ is nearly $I^\#(-Y)$. To make this more precise, we fix once and for all a diffeomorphism between the based 3-manifolds $(\Sigma_1\times S^1 , p_1)$ and $(T^3, p^\#)$, and lift this to a bundle isomorphism between the bundle defined by $\alpha_1\sqcup \eta_1\subset \Sigma_1\times S^1$ and that defined by $S^1\subset T^3$. These identifications induce an isomorphism
\[
	\Phi:I(-Y\# \Sigma_1\times S^1|-\Sigma_1)_{\alpha_1\sqcup \eta_1} \longrightarrow I^\#(-Y).
\]
We define the contact invariant in $I^\#(-Y)$ to be the image of the contact invariant under the above string of identifications:

\begin{definition}\label{contact-invariant}
	$\Theta^\#(Y,\xi) :=\left( \Phi\circ \Psi^{2}\circ \cdots \circ \Psi^g\circ F_{\phi}\right) \left( \theta(\mathcal{D}_R) \right)  \in I^\#(-Y)$.
\end{definition}

Our definition of $\Theta^\#(Y,\xi)$ may be extracted from the general instanton contact invariant of \cite{baldwin-sivek} as follows. First, in \cite{baldwin-sivek-nat}, to the sutured manifold $Y(p)=Y\setminus B(p)$, with a single suture on its boundary, the authors associate a projectively transitive system of $\C$-modules
\[
	\underline{\mathbf{SHI}}(Y(p)) = \left\{ \underline{SHI}(\mathcal{D}) \right\}_{\mathcal{D}}
\]
where the collection ranges over marked odd closures $\mathcal{D}$ of $Y(p)$, of genus at least 2; see Definition 2.12 of \cite{baldwin-sivek}. Each $\underline{SHI}(\mathcal{D})$ is an instanton Floer group defined over $\C$. What this means is that for any two marked odd closures $\mathcal{D}$ and $\mathcal{D}'$ of $Y(p)$, there is an isomorphism
\[
	\underline{\Psi}_{\mathcal{D},\mathcal{D}'}: \underline{SHI}(\mathcal{D})\longrightarrow \underline{SHI}(\mathcal{D}')
\]
well-defined up to multiplication by $\mathbb{C}^\times$ such that $\underline{\Psi}_{\mathcal{D}',\mathcal{D}''} \circ  \underline{\Psi}_{\mathcal{D},\mathcal{D}'} \sim  \underline{\Psi}_{\mathcal{D},\mathcal{D}''}$ and $\underline{\Psi}_{\mathcal{D},\mathcal{D}} \sim \text{id}$, where $\sim$ means equal up to multiplication by $\mathbb{C}^\times$.

Now, turning to our construction of $\Theta^\#(Y,\xi)$, we note that $(-Y\# R\times S^1,\alpha\sqcup \eta)$ and $(-Y\#\Sigma_i\times S^1,\alpha_i\sqcup \eta_i)$ naturally determine marked odd closures $\mathcal{D}_R$ and $\mathcal{D}_i$ of $-Y(p)$. Furthermore, we have, by definition, the identifications
\begin{align*}
	\underline{SHI}(\mathcal{D}_R) &=  I(-Y\# R\times S^1|-R)_{\alpha \sqcup \eta},\\
	 \underline{SHI}(\mathcal{D}_i) &=  I(-Y\# \Sigma_i\times S^1|-\Sigma_{i})_{\alpha_i \sqcup \eta_i}, 
\end{align*}
and the maps we defined above are instances of the maps $\underline{\Psi}_{\mathcal{D},\mathcal{D}'}$ from \cite{baldwin-sivek}, for $i\geq 3$:
\[
	F_{\phi}=\underline{\Psi}_{\mathcal{D}_R,\mathcal{D}_g} , \qquad \Psi^i = \underline{\Psi}_{\mathcal{D}_i,\mathcal{D}_{i-1}}. 
\]
In \cite{baldwin-sivek}, the instanton contact invariant is defined as a collection of elements $\{ \theta(\mathcal{D})\}_{\mathcal{D}}$ where $\theta(\mathcal{D})\in \underline{SHI}(\mathcal{D})$ and $\underline{\Psi}_{\mathcal{D},\mathcal{D}'}(\theta(\mathcal{D}))=\theta(\mathcal{D}')$ up to multiplication by $\mathbb{C}^\times$. Now our definition of $\Theta^\#(Y,\xi)$ may be described from the viewpoint of \cite{baldwin-sivek} as follows. Our choice of open book decomposition determines $\theta(\mathcal{D}_R)\in \underline{SHI}(\mathcal{D}_R)$ as constructed in \eqref{eq:thetadef}. We then have our natural genus 2 marked odd closure $\mathcal{D}_2$ of $-Y(p)$, and we consider
\[
	\theta(\mathcal{D}_2) = \underline{\Psi}_{\mathcal{D}_R,\mathcal{D}_2}(\theta(\mathcal{D}_R)) \in \underline{SHI}(\mathcal{D}_2) = I(-Y\# \Sigma_2\times S^1 | -\Sigma_2)_{\alpha_2\sqcup \eta_2}.
\]
Finally, as $I^\#(-Y)$ is defined using a genus 1 closure, which is not included in the construction of $\underline{\mathbf{SHI}}(-Y(p))$, we use one last isomorphism, $\Phi\circ \Psi^2$, to move to $I^\#(-Y)$:
\[
	\Theta^\#(Y,\xi)  = (\Phi\circ \Psi^2 )(\theta(\mathcal{D}_2) ) \in I^\#(-Y).
\]
Alternatively, we can add $I^\#(-Y)$ to the projectively transitive system $\underline{\mathbf{SHI}}(-Y(p))$ using the map $\Phi\circ \Psi^2$ and taking the projective transitive closure. As the marked odd closure $\mathcal{D}_2$ and the map $\Phi\circ \Psi^2$ are constructed rather independently of $Y$, the naturality properties established in \cite{baldwin-sivek-nat, baldwin-sivek} continue to hold for this slightly larger projectively transitive system. In particular, we summarize the following.

\begin{theorem}[cf. Theorem 1.1 of \cite{baldwin-sivek}]
	The element $\Theta^\#(Y,\xi)\in I^\#(-Y)$ is an invariant of the based contact manifold $(Y,\xi,p)$, well-defined up to multiplication by $\mathbb{C}^\times$. In particular, it is independent of the choice of open book decomposition of $(Y,\xi,p)$.
\end{theorem}

The following property of the contact invariant under Stein fillings will be the key ingredient in the proof of \Cref{dim-Isharp}.     
\begin{theorem}[cf. Theorem 1.6 of \cite{baldwin-sivek-su2}]\label{thm:baldwin-sivek-independence}
Let $W$ be a compact 4-manifold with boundary $Y$. Suppose that $W$ has Stein structures $J_1,\ldots,J_n$ which induce contact structures $\xi_1,\ldots,\xi_n$ on $Y$. If the Chern classes $c_1(J_1),\ldots,c_1(J_n)$ are distinct as elements in $H^2(W;\R)$, then
\[
	\Theta^\#(Y,\xi_1), \; \ldots,\; \Theta^\#(Y,\xi_n) \in I^\#(-Y)
\]
are linearly independent. In particular, the dimension of $I^\#(-Y)$ is at least $n$.
\end{theorem}

\subsection{A mod 2 invariant of 2-plane fields} In order to determine the mod 2 grading of the contact invariant $\Theta^\#$, we will need a mod 2 invariant of 2-plane fields first.  Let $\xi$ be an orientable 2-plane field over a closed, oriented 3-manifold $Y$. We define a mod 2 invariant of $\xi$ as follows. Let $(X,J)$ be an almost complex 4-manifold with $\partial X = Y$ such that $\xi = TY\cap J TY$. We then define
\begin{equation}\label{delta-4d}
	\delta(Y,\xi) := \frac{1}{2}(\chi(X)+\sigma(X) + b_1(Y) - 1) \pmod 2.
\end{equation}
If we view $X$ as a cobordism from the empty set to $Y$, then $\delta(Y,\xi)\equiv d(X)-1 \pmod 2$.

\begin{proposition}\label{delta-inv}
	The quantity $\delta(Y,\xi)\in\Z/2$ is an invariant of the homotopy class of the (unoriented) orientable 2-plane field $\xi$ on the oriented 3-manifold $Y$.
\end{proposition}
\begin{proof}

	Let $(X_0,J_0)$ and $(X_1,J_1)$ be almost complex 4-manifolds with $\partial X_0=Y$ and $\partial X_1=-Y$, each compatible with $\xi$. Glue $X_0$ and $X_1$ along $Y$ to form a closed 4-manifold $W$. The almost complex structures $J_0$ and $J_1$ fit together to define one on $W$, say $J$. We have
	\[
		c_1(W,J)^2 = 2\chi(W) + 3\sigma(W),
	\]
	as holds for any almost complex 4-manifold. Furthermore, $c_1(W,J)\equiv w_2(W) \pmod 2$ is charactersitic for the intersection pairing on $W$, and so $c_1(W,J)^2\equiv \sigma(W) \pmod 8$. Consquently $\chi(W)+\sigma(W)\equiv 0 \pmod 4$. Using additivity of euler characteristics and signatures, we then have $\delta(Y,\xi)+\delta(-Y,\xi)\equiv b_1(Y) - 1 \pmod 2$. As $(X_0,J_0)$ and $(X_1,J_1)$ were chosen independently, the quantities $\delta(Y,\xi)$ and $\delta(-Y,\xi)$ are independent of these choices and thus invariants of $(\pm Y,\xi)$.
	
	Finally, if $\xi'$ is homotopic to $\xi$, then for any $(X,J)$ as in the definition of $\delta(Y,\xi)$, we may attach $I\times Y$, to obtain a 4-manifold diffeomorphic to $X$, and extend the almost complex structure $J$ to one compatible with $\xi'$ by using a homotopy from $\xi$ to $\xi'$. Then we may use the same expression to compute $\delta(Y,\xi')$ as was used for $\delta(Y,\xi)$.
\end{proof}
Note in the course of the proof we have established that
\begin{equation}
	\delta(Y,\xi) + \delta(-Y,\xi) \equiv  b_1(Y) - 1 \pmod 2. \label{eq:dual}
\end{equation}
Suppose $(Y,\xi)$ and $(Y',\xi')$ arise as oriented contact 3-manifolds. Then we may form their connected sum, written $(Y\# Y', \xi \# \xi')$. The invariant $\delta(Y,\xi)$ satisfies
\begin{equation}
	\delta(Y\# Y',\xi\# \xi') = \delta(Y,\xi) + \delta(Y',\xi'). \label{eq:additivity}
\end{equation}
This is easily verified from the definition by taking the boundary sum of almost complex manifolds bounding each of the summands.

\begin{remark}\label{rmk:orientations}
The operation of reversing the orientation of a contact 3-manifold and the operation of taking connected sums do not commute. For example, consider the 3-sphere with its standard contact structure $(S^3,\xi_{\text{std}})$. Then $(-S^3,\xi_\text{std})\# (-S^3,\xi_\text{std})$ is not equivalent to the orientation-reversal of $(S^3,\xi_\text{std})\# (S^3,\xi_\text{std})=(S^3,\xi_\text{std})$, as the former has $\delta\equiv 0$ (by \eqref{eq:additivity}) and the latter has $\delta \equiv 1$ (by \eqref{eq:dual}).
\end{remark}

When the euler class of the 2-plane field $\xi$ is torsion, the quantity $\delta(Y,\xi)\in\Z/2$ is determined by other known invariants of the 2-plane field $\xi$. With notation as above, let
\[
	d_3(\xi) := \frac{1}{4}(c_1(X,J)^2 -3\sigma(X)-2\chi(X)) \in \Q.
\]
This was defined by Gompf, who showed in \cite[Theorem 4.5]{gompf} that $d_3$ is an invariant of the homotopy class of $\xi$. (Note that for $d_3(\xi)$ to make sense, the restriction of $c_1(X,J)$ to the boundary, that is the euler class of $\xi$, must be torsion.) Let us also recall the $\rho$ invariant of Spin$^c$ structures on a 3-manifold $Y$; this is a map $\rho:\text{Spin}^c(Y)\to \mathbb{Q}/2\Z$. For a Spin$^c$ structure $\mathfrak{t}\in \text{Spin}^c(Y)$, it is defined by
\[
	\rho(\mathfrak{t}) := \frac{1}{4}(c_1(\mathfrak{s})^2-\sigma(X)) \pmod 2
\]
where $(X,\mathfrak{s})$ is a Spin$^c$ 4-manifold bounded by $(Y,\mathfrak{t})$. Then we have
\begin{equation*}
	\delta(Y,\xi) \equiv \rho(\mathfrak{t}_\xi) - d_3(\xi) + \frac{1}{2}(b_1(Y)-1)  \pmod 2 
\end{equation*}
where $\mathfrak{t}_\xi$ is the Spin$^c$ structure determined by $\xi$.

\subsection{The mod 2 grading of the contact invariant}

\begin{proposition}\label{contact-grading}
	Suppose $\xi$ is a contact structure on a closed oriented 3-manifold $Y$. Then the instanton contact invariant $\Theta^\#(Y,\xi)\in I^\#(-Y)$ is homogeneous with respect to the $\Z/2$-grading, and is supported in degree $\delta(-Y,\xi)+1\in \Z/2$.
\end{proposition}

\begin{proof}
Since the contact invariant $\Theta^\#(Y,\xi)$ was defined in \Cref{contact-invariant} as the image of $\theta(\mathcal{D}_R)\in I(-Y\#R\times S^1 | -R)_{\alpha\sqcup\eta}$ under the composition of some change-of-genus excision isomorphisms, we begin by computing the $\Z/2$-grading of $\theta(\mathcal{D}_R)$. Baldwin and Sivek construct a compact symplectic 4-manifold $Z$ as a relatively minimal Lefschetz fibration over $D^2$ with boundary $-R\times_h S^1$ \cite[Lemma 3.5]{baldwin-sivek-su2}, and a cobordism $-V$ with oriented boundary $R\times_h S^1 \sqcup -(Y\# R\times S^1)$  \cite[Proposition 2.16]{baldwin-sivek-su2}. Then, according to \cite[Lemma 3.6]{baldwin-sivek-su2} the element $\theta(\mathcal{D}_R)$ can be realized as the image under $I(-V|-R)_{\nu}$ of the relative Donaldson invariant $I(Z|-R)_{\widetilde{\nu}}\in I(-R\times_h S^1 | -R)_{\alpha\sqcup\eta}$ associated to the 4-manifold $Z$, and restricted to the appropriate eigenspaces. Therefore, if $X$ is the compact 4-manifold
\[
	X = Z \cup_{-R\times_h S^1} (-V),
\]
then we may regard it as a cobordism $X: \emptyset \to -Y\# R\times S^1$, and the mod 2 grading of the relative invariant $\theta(\mathcal{D}_R) = I(-V|-R)_\nu\left( I(Z|-R)_{\widetilde \nu}\right)$ is given by
\begin{align*}
	d(X) &\equiv \frac{1}{2}\left(\chi(X)+\sigma(X)+b_1(-Y\# R\times S^1)+1\right)  \pmod 2.
\end{align*}
To relate $d(X)$ to $\delta(-Y,\xi)$ we must endow $X$ with an almost complex structure that restricts to $- (Y \# R \times S^1)$ with 2-plane field $\xi'$. Here $\xi'$ is as follows. Let $\xi_R$ be the 2-plane field on $R \times S^1$ given by the tangencies to $R$. Then $\xi'$ is the 2-plane field on the connected sum $(Y,\xi)\# (R\times S^1, \xi_R)$, and $(-(Y\# R \times S^1),\xi')$ is the result of reversing the orientation of the underlying 3-manifold. (Order matters: see \Cref{rmk:orientations}.)

Let $(R \times_h S^1,\tau_{R})$ be a 3-manifold with 2-plane field where $\tau_R$, like $\xi_R$, consists of the tangencies of the surface fibers. As $Z$ has the structure of a Lefschetz fibration, it admits an almost complex structure compatible with $\tau_R$ at the boundary. Thus to show that $X$ admits an almost complex structure of the required sort, it suffices to show that $-V$ has an almost complex structure that restricts to the prescribed plane fields $\xi'$ and $\tau_R$ at its boundary components. This is described in \Cref{lemma:enemy}.

Proceeding, by \eqref{eq:dual} we have
\begin{align*}
	d(X)  \equiv \delta(-(Y\# R\times S^1),\xi') + 1 & \equiv \delta(Y\# R\times S^1,\xi') + b_1(Y\# R\times S^1)\\
	 &  \equiv \delta(Y\# R\times S^1,\xi') + b_1(Y) + 1 \pmod 2 .
\end{align*}

Now we may use \eqref{eq:additivity} and \eqref{eq:dual} to compute
\begin{align*}
	d(X)  &  \equiv  \delta(Y,\xi) + \delta(R\times S^1,\xi_R) + b_1(Y) + 1 \\ & \equiv \delta(-Y,\xi) +  \delta(R\times S^1,\xi_R) \pmod 2 .
\end{align*}
We compute that $\delta(R\times S^1,\xi_R)\equiv 1 \pmod 2$ as follows: the 4-manifold $R\times D^2$ clearly has an almost complex structure which restricts to $\xi_R$ on the boundary, and thus
\begin{align*}
	\delta(S^1\times R,\xi_R) & \equiv \frac{1}{2}\left(\chi(R\times  D^2) + \sigma(R\times D^2) + b_1(R\times S^1)-1\right)\\
					      & \equiv \frac{1}{2}\left(2-2g + 0 + 2g + 1 - 1 \right) \equiv 1 \pmod 2.
\end{align*}
From this, we conclude that the $\Z/2$-grading of $\theta(\mathcal{D}_R)$ is given by
\[
	\text{gr}\left(\theta(\mathcal{D}_R)\right) = \delta(-Y,\xi) + 1 \pmod 2.
\]
Now, to compute the $\Z/2$-grading of the framed contact invariant $\Theta^\#(Y,\xi)\in I^\#(-Y)$ it suffices to compute the mod 2 degree of the map
\begin{equation}
	\Phi\circ \Psi^{2}\circ \cdots \circ \Psi^g\circ F_{\phi}.\label{eq:comp}
\end{equation}
The maps $\Phi$ and $F_{\phi}$ are induced by diffeomorphisms and have degree $0$. Recall that the map $\Psi^i$ is induced from $I(-W^i)_\nu$ by identifying $I(-\Sigma_2\times S^1|-\Sigma_2)_{\eta_2}$ with a copy of $\C$. In \Cref{lemma:surfacesupport} below we show that this copy of $\C$ is supported in grading $0$. Thus
\[
	\text{deg}( \Psi^i ) = d(-W^i) \pmod 2
\]
for $i=2,\ldots,g$. Next, recall that $W^i:Y_1\to Y_2$ where $Y_1=Y\# \Sigma_i\times S^1$ and $Y_2$ is the disjoint union of $ Y\# \Sigma_{i-1}\times S^1$ and $\Sigma_2\times  S^1$. Further, $W^i$ is formed by gluing the three pieces, as in \eqref{eq:w1}, \eqref{eq:w2}, \eqref{eq:w3}; it follows easily that $\chi(W^i)=0$ and $\sigma(W^i)=0$. Thus
\begin{align*}
	d(-W^i)  & \equiv \frac{1}{2}\left(\chi(-W^i) + \sigma(-W^i) +b_1(Y_2)-b_1(Y_1) + b_0(Y_2) - b_0(Y_1)  \right) \\
		     & \equiv  \frac{1}{2}\left(0 + 0 + (b_1(Y)+2i+4)-(b_1(Y)+2i+1) + (2) - (1)  \right) \\
		     & \equiv \frac{1}{2}( 4) \equiv 0 \pmod 2 .
\end{align*}
Thus the degree of $\Psi^i$ is $0 \pmod 2$, and consequently the composition \eqref{eq:comp} also has degree $0 \pmod 2$. Finally, we conclude that the mod 2 grading of $\Theta^\#(Y,\xi)$ is the same as that of $\theta(\mathcal{D}_R)$, completing the proof.
\end{proof}

\begin{lemma} \label{lemma:enemy}
	The cobordism $-V$ has an almost complex structure which restricts to the 2-plane fields $\xi\#\xi_R$ and $\tau_R$ at the boundary components.
\end{lemma}
We remark that the proof of \Cref{lemma:enemy} below is similar to that of \cite[Proposition 2.2]{baldwin-monoid}.  
\begin{proof}
The construction of the cobordism 
\[
	V: R \times_h S^1 \to Y\# R \times S^1
\]
is described by Baldwin-Sivek in \cite[Proposition 2.16]{baldwin-sivek-su2} and we begin our proof recalling the most important details of their construction. First, $V$ is the result of attaching 4-dimensional 2-handles to the outgoing end of $R \times_h S^1\times [0,1]$ along the closed curves
\[
	\gamma_i=(c_i\times\{1\}) \cup (\partial c_i \times [-1,1]) \cup (h(c_i)\times\{-1\}).
\]
The framing is given by the surface pushoff in $F:=\partial\left(S\times [-1,1]\right)$ as in \eqref{gamma}. To endow $V$ with an almost complex structure, notice that the 2-plane field $\tau_R$ restricted to $S\times [-1,1]\subset R\times_h S^1$ is homotopic to the contact structure on $S\times [-1,1]$ that is $[-1,1]$-invariant (with rounded corners), and that makes $F=\partial(S\times [-1,1])$ a convex surface with dividing set isotopic to $\Gamma_S=\partial S\times\{0\}$. We therefore assume that $\tau_R$  is in fact equal to this contact 2-plane field on a neighborhood of $S\times [-1,1] \subset R\times_h S^1$. Furthermore, by the Legendrian realization principle \cite{kanda,honda-tightI}, after a perturbation of the convex surface $F$, the curves $\gamma_i$ are Legendrian. A result of Kanda \cite{kanda} computes the contact framing relative to the surface framing as $-1/2$ times the number of times the curve $\gamma_i$ intersects the dividing set. In this case the curves intersect $\partial S\times\{0\}$ twice, so each $\gamma_i$ has contact framing $+1$. Thus, if $\mathbb{L}$ denotes the framed link obtained as the union of the $\gamma_i$ curves with $F$-framing, the cobordism $V$ is the 2-handle cobordism associated to $+1$ contact surgery on the Legendrian link $\mathbb{L}\subset R\times_h S^1$. 

Consider the handle decomposition for $(-V,Y\# R \times S^1)$ obtained as the dual handle decomposition of $(V,R\times_h S^1)$. That is, consider the handle decomposition for $(-V,Y\# R \times S^1)$ consisting of 2-handles attached along the belt spheres of the original 2-handles. As explained in \cite[Proposition 8]{ding-geiges}, this dual handle decomposition is the surgery cobordism of $-1$ contact surgery on a Legendrian link.  Here, the relevant 2-plane fields are unchanged on the two three-manifolds.  We emphasize our current viewpoint: we begin with $R\times_h S^1$ with the 2-plane field $\tau_R$ which is contact on a neighborhood of $S\times [-1,1]$; it is not important for us that we have not prescribed a contact structure on all of $R\times_h S^1$.   It is shown in \cite{weinstein} that $-V$ has an almost complex structure compatible with $\tau_R$ on $-R \times_h S^1$ on the boundary and some 2-plane field on $-(Y \# R \times S^1)$.  Note that this is the 2-plane field that arises from contact $+1$-surgery on $\mathbb{L}$.  Therefore, our goal is to identify this homotopy class of 2-plane fields with that of $\xi \# \xi_R$. 

To understand the result of contact $+1$-surgery on $\mathbb{L}$, decompose $R \times_h S^1 = R\times [-1,3]/\sim$ along $F$ to obtain the following two components:
\[
	P = S \times [-1,1], \qquad Q = (R\times_h S^1)\setminus \text{int}\left( S\times [-1,1]\right).
\]
Similarly, identifying $S^1=\R/2\pi\Z=[\pi/2,3\pi/2]\cup [-\pi/2,\pi/2]$, decompose the solid torus $D^2\times S^1$ attached along a tubular neighborhood of $\gamma_i$ into two pieces:
\[
H^-_i=D^2\times [\pi/2,3\pi/2],\qquad H^+_i=D^2\times [-\pi/2,\pi/2].
\]
The outgoing boundary component of $V$ is therefore the result of attaching 3-dimensional 2-handles to $P$ and $Q$, and then gluing the resulting pair of 3-manifolds with 2-sphere boundary together. The specifics of how these 2-handles are glued to each piece will be recounted below as needed.  The key observation is that cutting along $\gamma_i$ and attaching $H^-_i$ and $H^+_i$ as {\em contact} 2-handles is the same as doing {\em contact} $+1$-surgery on the Legendrian curves denoted by $\gamma_i$.  This can be seen by simply observing that both operations correspond to removing a standard neighborhood of $\gamma_i$ and gluing in a tight solid torus with two parallel dividing curves of the same slope.  

We must now identify the 2-plane field on $Y\# R\times S^1$ resulting from these contact 2-handle attachments. First, the plane field on $Y\setminus B^3\subset Y\# R\times S^1$ is obtained by attaching the 2-handles $H^-_i$ to $P = S \times [-1,1]$, now viewed as contact 2-handles. The result is $(Y\setminus B^3,\xi)$, by the very definition of what it means for $(S,h,\mathbf{c})$ to be an open book decomposition as in \Cref{sec::contactinvt}. Next, to identify the plane field on the $R\times S^1$ summand, write $Q'= (R\times S^1)\setminus \text{int}(S \times [-1,1])$ and recall that $S \times [-1,1]$ admits a product disk decomposition. Then $R\times S^1 \setminus B^3$ may be identified with the union of $Q'$ and 3-dimensional 2-handles attached along the curves $$\gamma_i'=c_i \times \{-1\} \cup \partial c_i \times [-1,1] \cup c_i\times \{1\},\;\;i\in\{1,\ldots,b_1(S)\}.$$ Furthermore, the diffeomorphism $\psi:Q\to Q'$ given by \begin{equation}\label{eq:maptpart} \psi(x,t)=\begin{cases} (h(x),-1)&\text{ if } t=-1\\ (x,t)&\text{ else} \end{cases}\end{equation} extends to a diffeomorphism %from \cite[Proposition 2.16]{baldwin-sivek-su2},
\begin{equation}
	Q \cup \bigcup_{i=1}^{b_1(S)} H^+_{i}  \to R\times S^1 \setminus B^3=Q' {\cup} \bigcup_{i=1}^{b_1(S)}D^2\times D^1. \label{eq:diffw2handles}
\end{equation}
with each $H^+_i$ attached to $Q$ along $\gamma_i$ as before. Note that under the diffeomorphism \eqref{eq:diffw2handles} the dividing set of $F$ is fixed and each $\gamma_i$ goes to the double of $c_i$, which we have denoted by $\gamma_i'$. See \Cref{fig:swdivset}.

Having arranged that $\tau_R$ restricted to a neighborhood of $S\times [-1,1]\subset R\times_h S^1$ is contact, the resulting plane field on the outgoing end $R\times S^1\setminus B^3\subset V$ is obtained from $(Q,\tau_R|_{Q})$ by attaching the $H^+_i$'s viewed as contact 2-handles. Via the diffeomorphism \eqref{eq:diffw2handles}, this is the same as the result of filling $(Q',\xi_R|_{Q'})$ with the contact 2-handles obtained from the product disk decomposition for $S\times[-1,1]$ in terms of the disks $c_i\times [-1,1]$. Therefore, to describe the resulting plane field it is enough to identify the contact structure $\eta$ on $S\times[-1,1]$ that induces on $S\times\{0\}$ the same characteristic foliation and dividing set as $\tau_R|_{S\times\{0\}}$. This contact structure is associated to a product disk decomposition and its existence is guaranteed by a result of Giroux~\cite[Sub-Lemma III.3.3]{giroux-91}. Furthermore, $\eta$ is unique up to isotopy by a result of Torisu~\cite{torisu} and is homotopic to the 2-plane field obtained as the tangent space of $S$ by a result of Honda-Kazez-Matic~\cite[Lemma 5.2]{HKM-hyperbolic}. We have therefore extended $(Q',\xi_R|_{Q'})$ to $S\times [-1,1]$ via $TS$ and we have thus shown that the 2-plane field obtained as the restriction of the almost complex structure of $V$ to the $R\times S^1$ summand is homotopic to $\xi_R$.
\end{proof}

\begin{figure}
\centering
\includegraphics[scale=0.75]{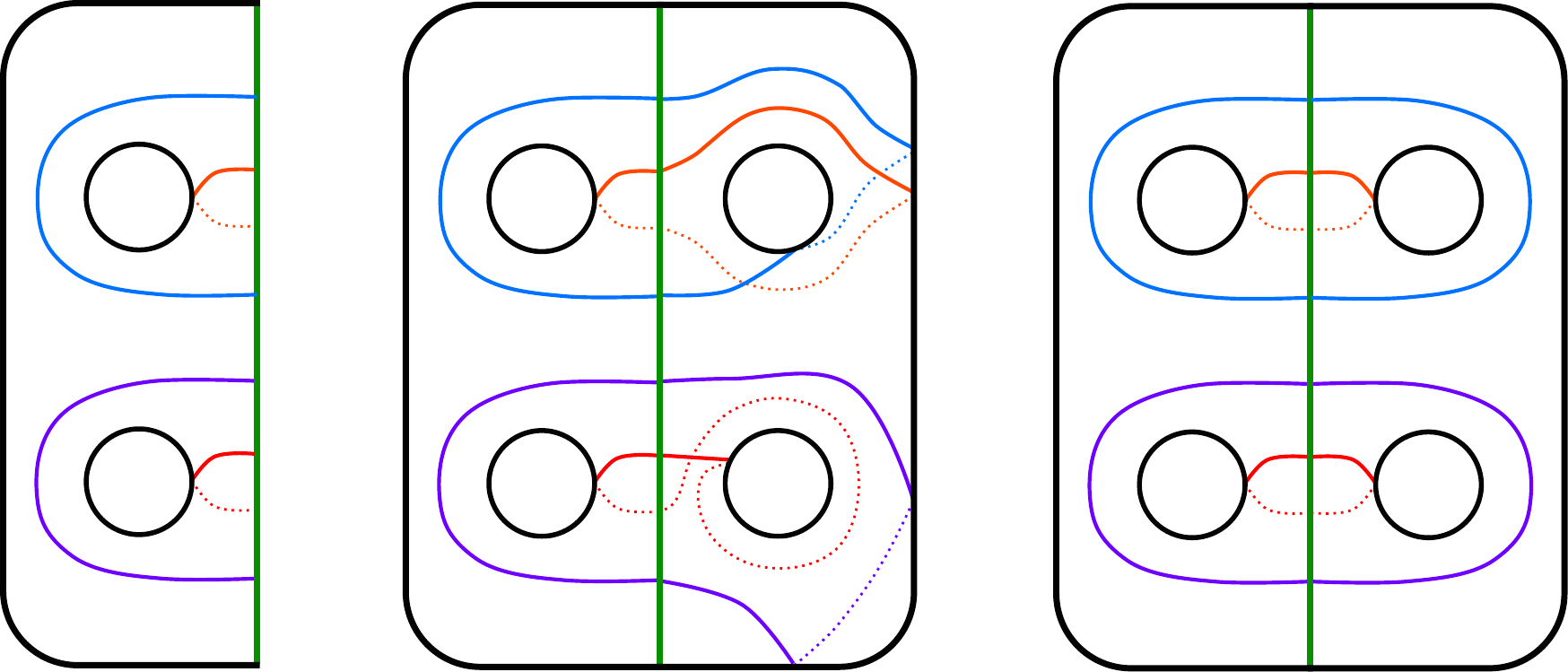}
\caption{On the left is the surface $S$ with the basis arcs $c_i$; in this example, $S$ has one boundary component (as we always arrange) and has genus 2. Center is the surface $F=\partial ( S\times [-1,1])$, viewed as the double of $S$, after rounding corners. This surface is a convex boundary for $S\times [-1,1]$ and has dividing set the boundary component of $S$. The loops $\gamma_i$ are shown. On the right is the result of applying the diffeomorphism \eqref{eq:diffw2handles}; the dividing set is fixed, while $\gamma_i$ is transformed into the double of $c_i$. } \label{fig:swdivset} 
\end{figure}

\begin{lemma}\label{lemma:surfacesupport} Let $\omega\subset \Sigma_g\times S^1$ be any 1-manifold such that $(\Sigma_g\times S^1,\omega)$ is admissible. Then $I(\Sigma_g\times S^1 | \Sigma_g)_{\omega}$ is 1-dimensional and is supported in mod 2 grading 0.
\end{lemma}

\begin{proof}
	It suffices to consider the following three cases:
	\begin{enumerate}
		\item $\omega=:\alpha$ is an $S^1$-factor;
		\item $\omega=:\eta\subset \Sigma_g \times \{0\}$ is an embedded closed curve with $[\eta]\neq 0 \in H_1(\Sigma_g;\Z/2)$;
		\item $\omega=\alpha\sqcup \eta$ is a union of the previous two cases.
	\end{enumerate}
	Indeed, the instanton homology depends only on $\Sigma_g\times S^1$ and $[\omega]\in H_1(\Sigma_g\times S^1;\Z/2)$, and any such class that comes from an admissible pair is mod 2 homologous to one of the three types of curves listed.
	
	Case (1) was dealt with in \Cref{lemma:1dim0mod2}.
	
	We next address case (3).  First, assume $g \geq 2$.  Consider $Y=S^3$ with the standard tight contact structure $\xi$. Consider an open book decomposition with pages $S$ of genus 1, and attach an auxiliary surface $T$ with genus $g-1\geq 1$. Then $R=S\cup T$ has genus $g\geq 2$. Recall we have a compact 4-manifold $X$ with boundary $-Y\# R\times S^1=-R\times S^1$. Then
\[
	\theta(\mathcal{D}_R) = I(X|-R)_{\widetilde{\nu}} \in  I(-R\times S^1 | -R)_{\alpha\sqcup\eta}.
\]
As this $\theta(\mathcal{D}_R)$ represents the contact invariant of the tight contact structure on $S^3$, it also generates the 1-dimensional space $I(-R\times S^1 | -R)_{\alpha\sqcup\eta}$. We also computed that the grading of $\theta(\mathcal{D}_R)$ is given by $\delta(-S^3,\xi)+1 \pmod 2$. The contact manifold $(S^3,\xi)$ bounds the Stein 4-ball $B^4\subset \C^2$, and from this we compute $\delta(-S^3,\xi)+1\equiv \delta(S^3,\xi)\equiv 0 \pmod 2$. Thus $\theta(\mathcal{D}_{R})$ and hence $I(-R\times S^1 | -R)_{\alpha\sqcup\eta}$ are supported in grading $0 \pmod 2$.

Now in the construction of $\theta(\mathcal{D}_{R})$, $\alpha$ is an $S^1$-factor of $R\times S^1$, but $\eta$ was chosen to be in the interior of $T$. However, after applying an automorphism of $R$ we may obtain any desired curve $\eta$ as in (2). This completes the case of (3) for genus $g\geq 2$. The genus 1 case is simpler, as all pairs $(T^3,[\omega])$ with $[\omega]\in H_1(T^3;\Z/2)$ that arise from admissible pairs $(T^3,\omega)$ are related by diffeomorphisms.

Finally, we consider case (2). The proof in \cite[Proposition 7.9]{km-excision} that $I(\Sigma_g\times S^1|\Sigma_g)_{\eta}$ is 1-dimensional uses an excision cobordism $W$ which induces an isomorphism
\[
	I(W)_\nu:I(\Sigma_g\times S^1|\Sigma_g)_{\alpha\sqcup \eta} \otimes  I(\Sigma_1 \times S^1|\Sigma_1)_{\alpha_1\sqcup \eta_1}\longrightarrow I(\Sigma_{g}\times S^1|\Sigma_g)_{\eta}.
\]
The domain of $I(W)_\nu$ is supported in grading $0 \pmod 2$ by case (3), and the degree of $I(-W)_\nu$ is computed similarly as in the proof of \Cref{contact-grading} to be $d(W)\equiv 0 \pmod 2$. This completes case (2).
\end{proof}

%-------------------------------------------------------------------------------------------------
%-------------------------------------------------------------------------------------------------
\section{Proof of \Cref{dim-Isharp}}\label{sec:dim-Isharp}
In this section we compute the framed instanton homology of 3--manifolds obtained as surgery on $S^3$ along instanton L-space knots.  The arguments below are standard and have appeared in various contexts in instanton, Heegaard, and monopole Floer homology, but we include them for completeness.  

We begin by giving a lower bound on the dimension of $\Isharp$ for negative surgeries on a knot with $\maxsl  \geq 0$.  This first result is a slight refinement of \cite[Proof of Theorem 1.10]{baldwin-sivek-su2}.
\begin{proposition}\label{lower-bound}
Let $K$ be a knot in $S^3$ with $\maxsl(K) = s \geq 0$ and let $n > 0$.  Then the dimension of $\Isharp (S^3_{-n}(K))$ in degree 0 is at least $s + n$, and the dimension of $\Isharp(S^3_{-n}(K))$ in degree 1 is at least $s$.  
\end{proposition}

Before proving the proposition, we need a bit more background.  We fix an orientation of $K$ throughout.  Let $W_{-n}(K)$ denote the trace of the surgery, that is  the result of attaching a $-n$-framed 2-handle to $B^4$, and let $\alpha$ denote the Poincar\'e dual of the co-core of the 2-handle, with orientation induced by that of $K$. Then, if $\Lambda$ is a Legendrian representative of $K$ with $tb(\Lambda) = -n+1$, the trace $W_{-n}(K)$ admits a Stein structure $J_\Lambda$ with $c_1(J_\Lambda) = r(\Lambda) \alpha$ by \cite[Proposition 2.3]{gompf}, where $r(\Lambda)$ is the rotation number. Note that if we reverse orientation on $\Lambda$, the sign of $r(\Lambda)$ changes, as does that of $\alpha$, so $c_1(J_\Lambda)$ does not change.  Recall that a Legendrian knot $\Lambda$ admits a positive (resp. negative) stabilization $\Lambda_+$ (resp. $\Lambda_-$) with $tb(\Lambda_\pm)=tb(\Lambda)-1$ and $r(\Lambda_\pm)=r(\Lambda)\pm 1$. 

\begin{proof}
We begin with the case that $n \gg 0$.  Choose a Legendrian representative $\Lambda'$ of $K$ with $tb(\Lambda') + r(\Lambda') = s$, and notice that $s$ is actually at least one since $tb + r$ is always odd.  Further, $r(\Lambda') \geq 0$, since otherwise a transverse pushoff of $\Lambda'$ would have self-linking number greater than $s$.  It is possible to obtain a Legendrian representation $\Lambda$ with $tb(\Lambda) = -n +1 $ by applying $tb(\Lambda') + n - 1$ successive stabilizations of any sign. Consequently, there are $tb(\Lambda') + n$ different ways of stabilizing $\Lambda'$, each yielding a different rotation number.  Then the values of these rotation numbers are
\[
r(\Lambda') - tb(\Lambda') - n + 1 + 2k, \  0 \leq k \leq tb(\Lambda') + n - 1.
\]
Consider this collection of $tb(\Lambda') + n - 1$ Stein structures induced on $W_{-n}(K)$, together with their conjugates.  For $n \gg 0$, this produces a set of at least $s + n$ Stein structures whose canonical classes are all distinct in $H^2(W_{-n}(K);\mathbb{R})$.  Consequently, by \Cref{thm:baldwin-sivek-independence}, the associated contact elements $\Theta^\#(\xi_1), \ldots, \Theta^\#(\xi_{n+s})$ are linearly independent elements in $\Isharp(-S^3_{-n}(K))$.  

Using \Cref{contact-grading}, we compute that the grading of $\Theta^\#(\xi_i)$ is $0$ for all $1 \leq i \leq n+s$, and thus the dimension of $\Isharp(S^3_{-n}(K))$ in degree 0 is at least $n + s$, since orientation-reversal preserves the graded dimension of instanton Floer homology for rational homology spheres.  However, by \cite[Corollary 1.4]{scaduto}, we have $|\chi(\Isharp(S^3_{-n}(K)))| = n$.  This now gives the required lower bound on the dimension of $\Isharp$ in degree 1.   

The case for arbitrary $n > 0$ follows by induction, applying the surgery exact triangle relating the instanton homologies of $S^3$, $S^3_{-n-1}(K)$ and $S^3_{-n}(K)$, together with the euler characteristic constraint.  
\end{proof}

\begin{remark}
Note that the above argument can be repeated verbatim for Heegaard Floer homology with $\Z/2$ coefficients (replacing the work of Baldwin-Sivek on the linear independence of the contact classes with the earlier work of Plamenevskaya \cite{plamenevskaya}).  To compute the Heegaard Floer homology of surgery on an L-space knot with $\C$ coefficients, one may apply the work of \cite{ozsvath-szabo-lens} or \cite{ozsvath-szabo-integer}, both of which hold over $\C$. 
\end{remark}

\begin{proof}[Proof of \Cref{dim-Isharp}] First, let $m$ be a positive integer such that $S^3_m(K)$ is an L-space.  As $\chi(\Isharp(S^3_m(K))) = m$, then $\Isharp(S^3_m(K)) \cong \C^m_{(0)}$. Then, to prove the theorem we will consider surgery coefficients $n$ in three different intervals: $n>m$, $2g(K) - 1 \leq n < m$, and $n<2g(K)-1$. We will additionally take advantage of the fact that the exact triangle in \Cref{exact-triangle} implies $$\dim \Isharp(S^3_i(K)) \leq \dim \Isharp(S^3_j(K)) +\dim \Isharp(S^3_k(K)),$$ for $(i,j,k)$ any permutation of the triple $(\infty,n,n+1)$.

For the case $n>m$, proceed inductively considering the surgery exact triangle corresponding to the triple $(S^3, S^3_{n}(K),S^3_{n+1}(K))$, with base case given by $(S^3, S^3_{m}(K),S^3_{m+1}(K))$. Then, the inequality $$\dim\Isharp(S^3_{n+1}(K))\leq \dim\Isharp(S^3_n(K))+\dim\Isharp(S^3)=n+1,$$ together with the relation $\chi\left(\Isharp(S^3_{n+1}(K))\right)=n+1$ shows that $\dim\Isharp_0(S^3_{n+1}(K))= n+1$ and $\dim\Isharp_1(S^3_{n+1}(K))= 0$ as sought.

Next, if there exists an integer $n$ satisfying $2g(K) - 1 \leq n < m$, we consider the surgery exact triangles corresponding to the triples $(S^3, S^3_{n-1}(K),S^3_{n}(K))$, with base case given by $(S^3, S^3_{m-1}(K),S^3_{m}(K))$. Let $W_{n}(K)$ denote the result of attaching an $n$-framed 2-handle to $B^4$, and $F$ denote the oriented closed surface obtained as the union of a Seifert surface for $K$ in $S^3$ and the core of the 2--handle.  Then, the self-intersection of $F$ satisfies  $$F\cdot F= n >2g(K) -2.$$ The adjunction inequality from \cite[Theorem 1.1]{kronheimer-mrowka-embedded2} then implies that the cobordism map from $\Isharp(S^3)$ to $\Isharp(S^3_{n}(K))$ vanishes and so $$\dim\Isharp(S^3_n(K)) = \dim\Isharp(S^3_{n-1}(K))+\dim\Isharp(S^3),$$ which together with the euler characteristic relation shows that $S^3_{n}(K)$ is an L-space. We then see that $S^3_n(K)$ is an L-space if $n \geq 2g(K) - 1$.

Now, for $n<2g(K) - 1$, a combination of the surgery exact triangle for the triple $(S^3, S^3_{n-1}(K),S^3_{n}(K))$ and the fact that $S^3_{2g(K) - 1}(K)$ is an L-space together with an induction that decreases the index gives the desired upper bounds on the graded dimension of $\Isharp(S^3_{n-1}(K))$.  On the other hand, applying the surgery exact triangle in conjunction with \Cref{lower-bound} and an induction that increases the index gives the opposite inequality for the graded dimension of the framed instanton Floer homology of the desired surgery.  This completes the proof.
\end{proof}

\begin{comment}
\begin{remark}
It is natural to ask how \Cref{dim-Isharp} extends to rational surgeries.  The same arguments can be easily used to show that 
\[
\dim \Isharp(S^3_{1/n}(T_{p,q})) \geq \begin{cases} 4ng - 2n - 1 & n > 0, \\ -4ng - 1 & n < 0,  \end{cases}   
\]
using Legendrian surgery descriptions of these manifolds (with several components).  However, the dimension of $\HFhat$ is significantly bigger for these manifolds.  Perhaps a more clever application of the surgery exact triangle would allow for the correct bounds.  
\end{remark}
\end{comment}

%-------------------------------------------------------------------------------------------------
%-------------------------------------------------------------------------------------------------
\section{Mod 4 gradings in Framed Instanton Homology}\label{sec:mod4}
In this section, we revisit our previous work in the setting of the $\Z/4$-grading on instanton Floer homology.  Recall that the $\Z/2$-grading on $I^\#(Y)$ may be lifted to an absolute $\Z/4$-grading \cite[Section 7.3]{scaduto}. 
For a cobordism $W$ equipped with a closed surface $S$, the induced map on $I^\#$ is homogeneous with respect to this grading, and in \cite[Proposition 7.1]{scaduto}, the shift is computed to be 
\[
	\deg I^\#(W)_S = d(W) + 2[S]\cdot[S] \pmod 4.
\]
In particular, if $W$ is {\emph{spin}}, then the degree is given by $d(W) \pmod 4$.

The maps in the triangle \Cref{exact-triangle} are not all induced by cobordisms with mod 2 nullhomologous surfaces. In particular, their $\Z/4$-degrees are not determined by the integers \eqref{eq:ddeg} alone. Instead, by \cite[Section 7.5]{scaduto}, we have 

\begin{lemma}\label{deg-triangle}
 In the exact triangle \eqref{exact-triangle}, exactly one of the 2-handle cobordisms is not spin. For a spin cobordism, $W$, the degree shift agrees with $d(W)\pmod 4$. The degree of the non-spin cobordism map is determined by the condition that the sum of the degrees of the three maps is congruent to $-1 \pmod 4$.
\end{lemma}

In this triangle, the non-spin cobordism is $S^3\to S^3_n(K)$ when $n$ is odd, and $S^3_{n+1}(K)\to S^3$ when $n$ is even. The numbers $d(W)$ are easily computed and we obtain the following.

\begin{corollary}\label{deg-triangle-S3}
	Let $n\in \Z $. The $\Z/4$-degrees of the cobordism maps in the exact triangle for $(S^3,S_n^3(K),S_{n+1}^3(K))$ depend on $n$ in the following way:

\begin{center}
    \begin{tikzpicture}[scale=1.87]
        \node (A) at (0,0) {$I^\#(S_n^3(K))$};
        \node (B) at (2,0) {$I^\#(S_{n+1}^3(K))$};
        \node (C) at (1,-1) {$I^\#(S^3)$};
        \node (D) at (1,-.45) {$n\geq 2$};
        \node (D) at (1,-.25) {$n$ even};
        \path[->,font=\scriptsize]
        (A) edge node[above]{degree 0} (B)
        (B) edge node[right]{\;\;degree 2} (C)
        (C) edge node[left]{degree 1\;\;} (A);
    \end{tikzpicture}
\quad
    \begin{tikzpicture}[scale=1.8]
        \node (A) at (0,0) {$I^\#(S_n^3(K))$};
        \node (B) at (2,0) {$I^\#(S_{n+1}^3(K))$};
        \node (C) at (1,-1) {$I^\#(S^3)$};
        \node (D) at (1,-.45) {$n\geq 1$};
        \node (D) at (1,-.25) {$n$ odd};
        \path[->,font=\scriptsize]
        (A) edge node[above]{degree 0} (B)
        (B) edge node[right]{\;\;degree 0} (C)
        (C) edge node[left]{degree 3\;\;} (A);
    \end{tikzpicture}
\end{center}
\begin{center}
    \begin{tikzpicture}[scale=1.8]
        \node (A) at (0,0) {$I^\#(S_{0}^3(K))$};
        \node (B) at (2,0) {$I^\#(S_{1}^3(K))$};
        \node (C) at (1,-1) {$I^\#(S^3)$};
        \node (D) at (1,-.35) {$n=0$};
        \path[->,font=\scriptsize]
        (A) edge node[above]{degree 2} (B)
        (B) edge node[right]{\;\;degree 2} (C)
        (C) edge node[left]{degree 3\;\;} (A);
    \end{tikzpicture}
\quad
    \begin{tikzpicture}[scale=1.8]
        \node (A) at (0,0) {$I^\#(S_{-1}^3(K))$};
        \node (B) at (2,0) {$I^\#(S_{0}^3(K))$};
        \node (C) at (1,-1) {$I^\#(S^3)$};
        \node (D) at (1,-.35) {$n=-1$};
        \path[->,font=\scriptsize]
        (A) edge node[above]{degree 3} (B)
        (B) edge node[right]{\;\;degree 2} (C)
        (C) edge node[left]{degree 2\;\;} (A);
    \end{tikzpicture}
\end{center}
\begin{center}
    \begin{tikzpicture}[scale=1.87]
        \node (A) at (0,0) {$I^\#(S_n^3(K))$};
        \node (B) at (2,0) {$I^\#(S_{n+1}^3(K))$};
        \node (C) at (1,-1) {$I^\#(S^3)$};
        \node (D) at (1,-.45) {$n\leq -2$};
        \node (D) at (1,-.25) {$n$ even};
        \path[->,font=\scriptsize]
        (A) edge node[above]{degree 0} (B)
        (B) edge node[right]{\;\;degree 3} (C)
        (C) edge node[left]{degree 0\;\;} (A);
    \end{tikzpicture}
\quad
    \begin{tikzpicture}[scale=1.8]
        \node (A) at (0,0) {$I^\#(S_n^3(K))$};
        \node (B) at (2,0) {$I^\#(S_{n+1}^3(K))$};
        \node (C) at (1,-1) {$I^\#(S^3)$};
        \node (D) at (1,-.45) {$n\leq -3$};
        \node (D) at (1,-.25) {$n$ odd};
        \path[->,font=\scriptsize]
        (A) edge node[above]{degree 0} (B)
        (B) edge node[right]{\;\;degree 1} (C)
        (C) edge node[left]{degree 2\;\;} (A);
    \end{tikzpicture}
\end{center}
\end{corollary}

Use the shorthand $(d_0,d_1,d_2,d_3)$ for a complex $\Z/4$-graded vector space with dimension $d_i$ in grading $i \pmod 4$. We obtain the following refinement of \Cref{dim-Isharp}:

\begin{corollary}\label{Z4-grading} Let $K$ be a non-trivial knot in $S^3$ such that $S^3_n(K)$ is a lens space for a positive integer $n$. Then, setting $g=g(K)$, as $\Z/4$-graded vector spaces we have
\[
I^\#(S_{n}^3(K)) \cong
\begin{cases}
 (g+ \lfloor -n/2 \rfloor,g-1,g+\lfloor -(n+1)/2 \rfloor,g) &n \leq -1 \\
 (g-1,g-1,g,g) & n = 0\\
(g,g-\lceil n/2\rceil,g-1,g-1-\lfloor n/2 \rfloor)  & 1 \leq n \leq 2g-1 \\
(\lceil (n+1)/2\rceil, 0, \lfloor(n-1)/2\rfloor,0)  & n \geq 2g - 1.
\end{cases}
\]
\end{corollary}

Note that this agrees with the computations of $I^\#(\Sigma(2,3,7))$ and $I^\#(\Sigma(2,3,5))$ from \cite[Corollary 1.6, 1.7]{scaduto}, which were obtained using an entirely different method.  The extra constraint that ``L-space'' be replaced by ``lens space'' is only because we do not know the $\Z/4$-gradings for a general L-space.

\begin{proof} For a lens space $Y$ with $|H_1(Y;\Z)|=n$ we have $$I^\#(Y) \cong (\lceil (n+1)/2\rceil, 0, \lfloor(n-1)/2\rfloor,0).$$ Since $K$ admits a positive lens space surgery, it is a Heegaard Floer L-space knot, and hence is fibered and strongly quasipositive.  Therefore, $K$ realizes the transverse Bennequin bound.  The result now follows by repeating the arguments in \Cref{dim-Isharp} while using \Cref{deg-triangle-S3} to track the $\Z/4$-gradings in the surgery exact triangle.  
\end{proof}

\begin{question} 
Suppose that $\dim I^\#(Y)=|H_1(Y;\Z)|$, that is,  $Y$ is an instanton L-space. Is it necessarily true that $I^\#(Y) \cong (\lceil (n+1)/2\rceil, 0, \lfloor(n-1)/2\rfloor,0)$ where $n=|H_1(Y;\Z)|$?
\end{question}

Finally, we prove \Cref{contact-gradings-Z4}, namely that the contact invariant is not necessarily homogeneous with respect to the $\Z/4$-grading.  
\begin{proof}[Proof of \Cref{contact-gradings-Z4}]
	Consider $Y=\Sigma(2,3,7)=S_{-1}^3(T_{2,3})$. Let $W$ be the surgery cobordism from $-Y$ to $S^3$. We have shown that there are contact structures $\xi_1$ and $\xi_2$ induced by different Stein structures on $W$, such that $\Theta^\#(\xi_1)$ and $\Theta^\#(\xi_2)$ are linearly independent in $I^\#(-Y)$. Each of these is mapped to the generator $\Theta^\#(\xi_{std})\in I^\#(S^3)$ by the cobordism map $I^\#(W)$. Suppose $\Theta^\#(\xi_1)$ and $\Theta^\#(\xi_2)$ are homogeneous in the $\Z/4$-grading. Then the gradings must be the same, as the cobordism map $I^\#(W)$ is $\Z/4$-homogeneous. However, as computed in \Cref{Z4-grading}, $I^\#(-Y)$ has rank 0 or 1 in each $\Z/4$-grading, a contradiction.
\end{proof}

%-------------------------------------------------------------------------------------------------
%-------------------------------------------------------------------------------------------------
\section{An Example}\label{sec:trefoil}
In this section we will use the methodology implemented in the proof of \Cref{dim-Isharp} to show that \Cref{IsharpisHFhat} also holds for $1/n$-surgery along the right-handed trefoil for $n > 0$. 

\begin{proposition}[Corollary 1.7 of \cite{scaduto}]\label{thm:1ntrefoil}
For $n > 0$, the $\Z/2$-graded framed instanton Floer homology of $1/n$-surgery on the right-handed trefoil is given by 
\begin{equation}
\Isharp(S^3_{1/n}(T_{2,3}))\cong\C^n_{(0)}\oplus\C^{n-1}_{(1)}.\label{eq:oneovern}
\end{equation}
In particular, this is isomorphic to $\HFhat(S^3_{1/n}(T_{2,3}))$.  
\end{proposition}

While we do not do it here, the $\Z/4$-gradings are easily computed as well.  

\begin{proof}
We start by realizing $S^3_{1/n}(T_{2,3})$ as $(0,-n)$-surgery on $S^3$ along $T_{2,3}\sqcup \mu$, where $\mu$ is a meridian of the trefoil. Since we have computed $\Isharp(S^3_{1}(T_{2,3}))$ previously, we only consider the case $n \geq 2$.  To incorporate the contact geometry of these manifolds, we first realize $S^3_{1/n}(T_{2,3})$ as several distinct Legendrian surgeries along suitable Legendrian representatives of $T_{2,3}\sqcup \mu$ with Thurston-Bennequin number $(1, -n+1)$.

\begin{figure}\xymatrix{
&{\begin{array}{c}\def\svgwidth{0.25\textwidth}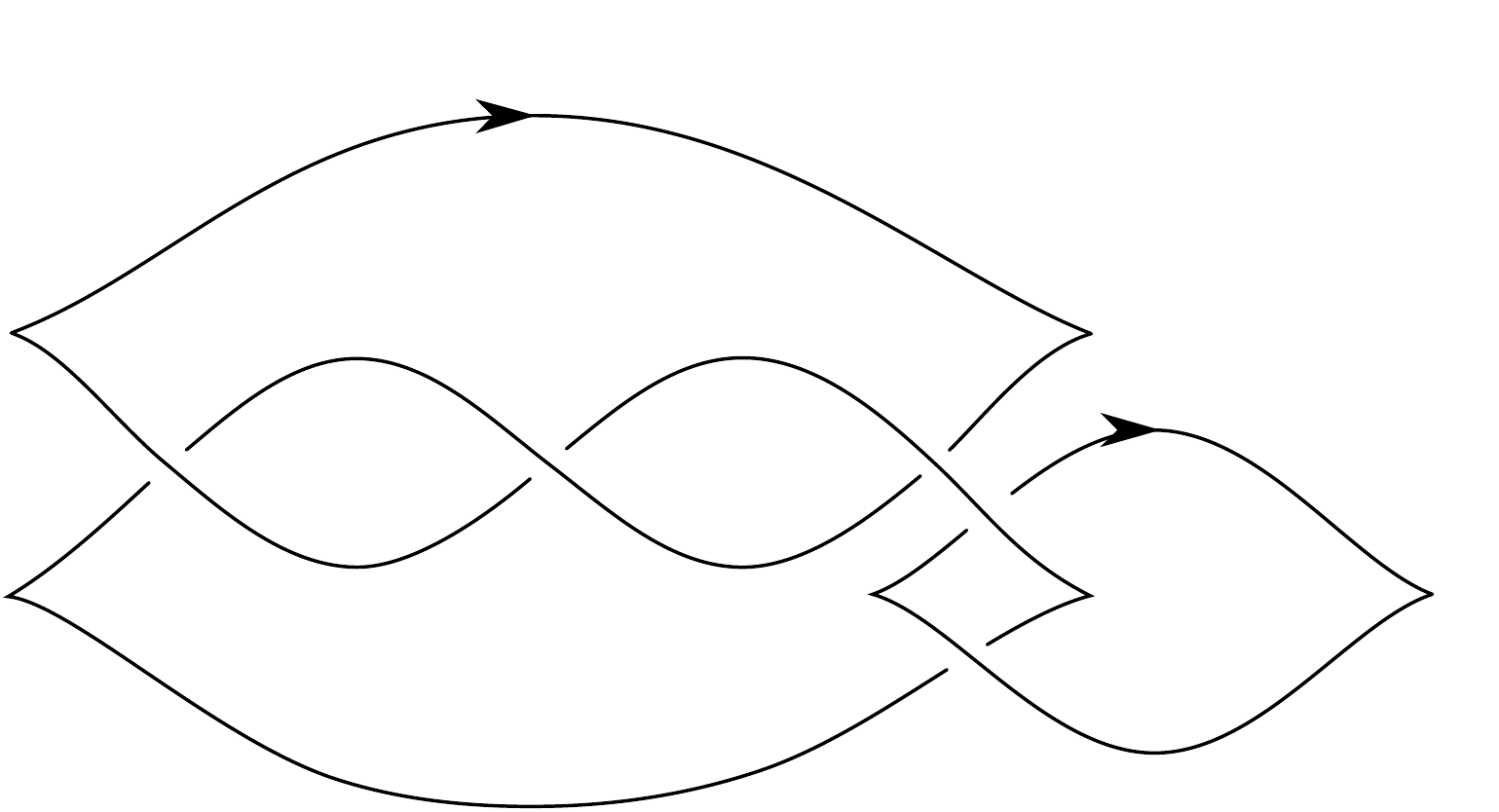\\tb=(0,-1)\\r=(0,0)\end{array}}\ar[ld]^{(-)}\ar[rd]_{(+)}&\\
{\begin{array}{c}\def\svgwidth{0.25\textwidth}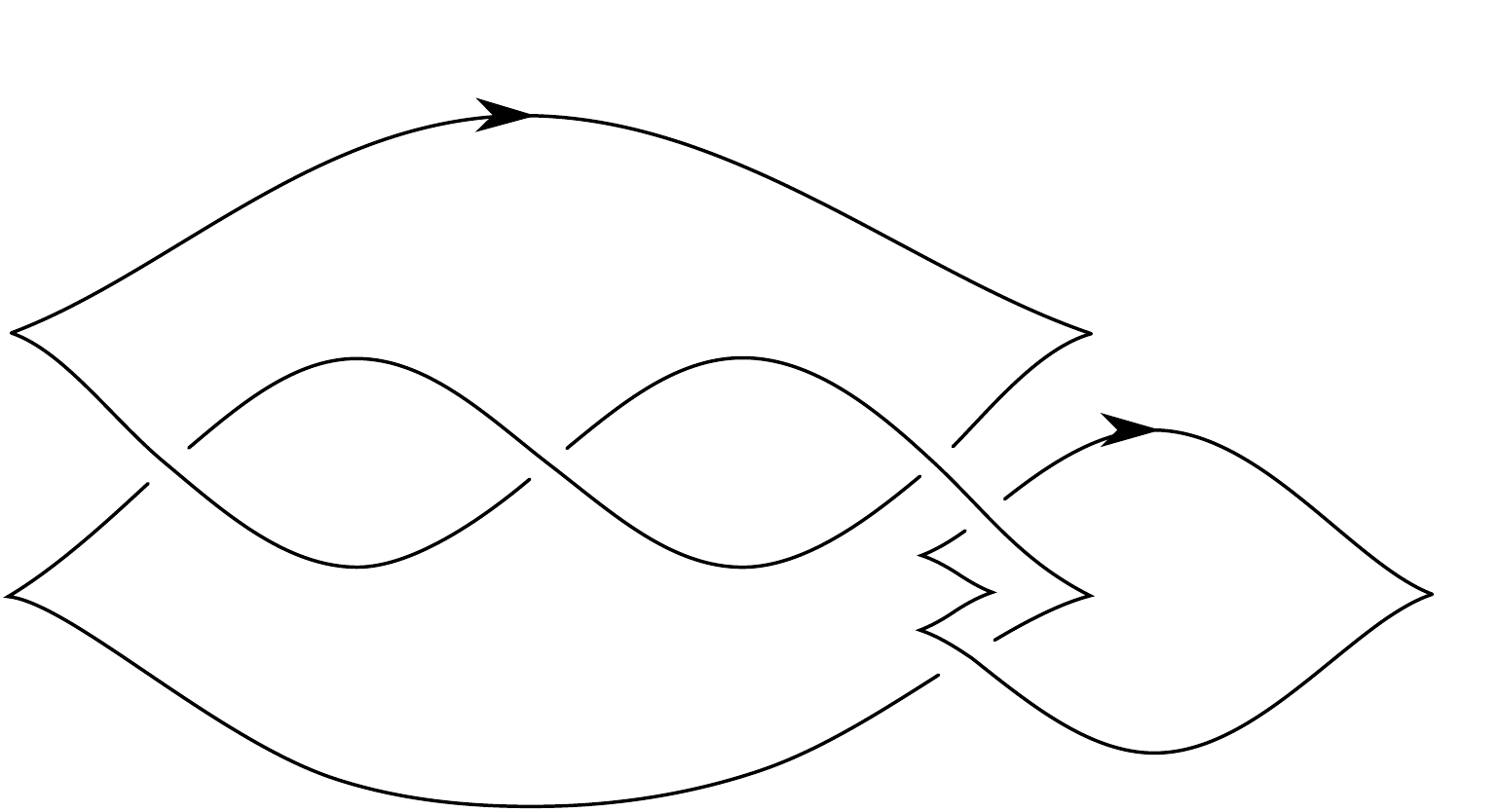\\tb=(0,-2)\\r=(0,-1)\end{array}}&&{\begin{array}{c}\def\svgwidth{0.25\textwidth}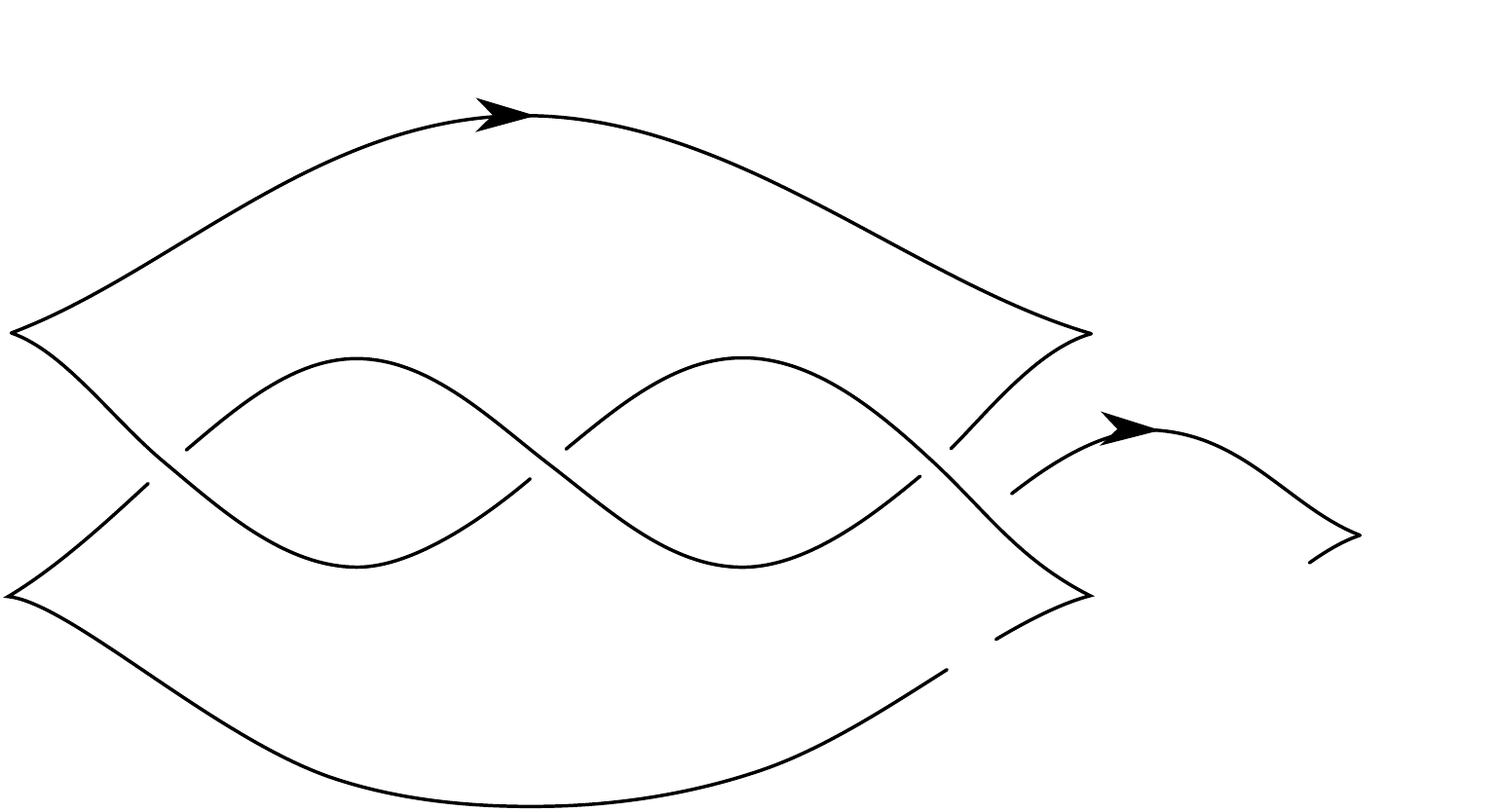\\tb=(0,-2)\\r=(0,1)\end{array}}\\
}

\caption{Legendrian surgery diagrams for $S^3_{1/3}(T_{2,3})$ using stabilizations on the unknotted component of $T_{2,3}\sqcup \mu$.}\label{mountain}
\end{figure}

We begin by establishing a lower bound on the dimension of $\Isharp(S^3_{1/n}(T_{2,3}))$.  Continuing as in \Cref{mountain}, we see that there are $n-1$ distinct rotation numbers arising from these Legendrian surgery representations.  It follows that $S^3_{1/n}(T_{2,3})$ has $n-1$ Stein fillings with distinct first Chern classes.  The grading of the contact element associated to each Stein filling is the same, since they are completely determined by cohomological invariants of $X$ and can be computed to be $$\delta=\frac{1}{2}\left(\chi(X)+\sigma(X)+b_1(\partial X)-1\right)=1.$$ 
Therefore, the dimension of $\Isharp(S^3_{1/n}(T_{2,3}))$ in odd $\Z/2$-grading is at least $n-1$.  Since the euler characteristic of the framed instanton homology of a homology sphere is 1, we have
\begin{equation}\label{eq:1nlowerbound}
\dim \Isharp(S^3_{1/n}(T_{2,3}))  \geq 2n-1.
\end{equation}  

We will now show that $2n-1$ is also an upper bound for the dimension of $\Isharp(S^3_{1/n}(T_{2,3}))$, thus obtaining \eqref{eq:oneovern}. It is well known that this agrees with the $\Z/2$-graded hat Heegaard Floer homology of these manifolds, thus completing the proof of \Cref{thm:1ntrefoil}.  (This is easily deduced from \cite[Proposition 9.6]{ozsvath-szabo-rational} or \cite[Equation (1)]{tweedy-brieskorn}, for instance.)  

The argument now proceeds by induction.  We have already computed the framed instanton homology of $+1$-surgery on the right-handed trefoil, so we move to the induction step.  Note that $1/n$-surgery on the right-handed trefoil fits into the the surgery triple $(1/n,1/(n-1),0)$.  Since the dimension of the framed instanton homology of 0-surgery on the right-handed trefoil is 2, the total dimension of $\Isharp(S^3_{1/n}(T_{2,3}))$ can only be at most 2 greater than the total dimension of that for $1/(n-1)$-surgery.  Since \eqref{eq:oneovern} holds for $1/(n-1)$ surgery by assumption, \eqref{eq:1nlowerbound} is necessarily equality, and the proof is complete.
\end{proof}

\nocite{*}
\bibliographystyle{amsplain}
\bibliography{references}

\end{document}